\documentclass[12pt, reqno]{amsart}

\usepackage[usenames,dvipsnames]{color}

\usepackage{amssymb, amsmath, amsthm,mathtools}
\usepackage[colorlinks=true, linkcolor = blue, citecolor= Green]{hyperref}
\usepackage[alphabetic,lite]{amsrefs}
\usepackage{amscd}   
\usepackage{fullpage}
\usepackage[all]{xy} 
\usepackage{faktor}
\usepackage{setspace}
\usepackage{enumitem}

\DeclareFontEncoding{OT2}{}{} 

\usepackage{tikz-cd}

\newtheorem{lemma}{Lemma}[section]
\newtheorem{theorem}[lemma]{Theorem}
\newtheorem{proposition}[lemma]{Proposition}
\newtheorem{prop}[lemma]{Proposition}
\newtheorem{cor}[lemma]{Corollary}
\newtheorem{conj}[lemma]{Conjecture}

\newtheorem{claim*}{Claim}
\newtheorem{thm}[lemma]{Theorem}

\theoremstyle{definition}
\newtheorem{remark}[lemma]{Remark}

\newtheorem{definition}[lemma]{Definition}

\newcommand{\PP}{{\mathbb P}}
\newcommand{\C}{{\mathbb C}}
\newcommand{\F}{{\mathbb F}}
\newcommand{\Q}{{\mathbb Q}}

\newcommand{\Z}{{\mathbb Z}}
\newcommand{\NN}{{\mathbb N}}

\newcommand{\Qbar}{{\overline{\Q}}}
\newcommand{\Zhat}{{\hat{\Z}}}

\newcommand{\kbar}{{\overline{k}}}

\newcommand{\kk}{{\mathbf k}}

\newcommand{\calE}{{\mathcal E}}

\newcommand{\frakm}{{\mathfrak m}}

\newcommand{\mm}{{\mathfrak m}}

\usepackage[mathscr]{euscript}

\DeclareMathOperator{\Sym}{Sym}

\DeclareMathOperator{\lcm}{lcm}

\DeclareMathOperator{\im}{im}

\DeclareMathOperator{\Aut}{Aut}
\DeclareMathOperator{\Gal}{Gal}

\DeclareMathOperator{\divv}{div}

\DeclareMathOperator{\Pic}{Pic}
\DeclareMathOperator{\Jac}{Jac}

\DeclareMathOperator{\PGL}{PGL}

\DeclareMathOperator{\Mat}{M}

\DeclareMathOperator{\SL}{SL}
\DeclareMathOperator{\PSL}{PSL}
\DeclareMathOperator{\GL}{GL}

\DeclareMathOperator{\Gl}{GL}
\DeclareMathOperator{\gon}{gon}
\DeclareMathOperator{\Supp}{Supp}

\newcommand{\isom}{\cong}

\newcommand{\defPisolated}{\PP^1\textup{\defi{-isolated}}}

\newcommand{\Pisolated}{\PP^1\textup{-isolated}}

\DeclareMathOperator{\defAVisolated}{\defi{AV-isolated}}
\DeclareMathOperator{\AVisolated}{{AV-isolated}}

\numberwithin{equation}{section}
\numberwithin{table}{section}

\newcommand{\defi}[1]{\textsf{#1}} 

\newcommand{\BB}[1]{B^1_{#1}}

\newcommand{\Snoell}{S-\{\ell\}}

\title{On the level of modular curves that give rise to isolated $j$-invariants}

\author{Abbey Bourdon}
\address{Wake Forest University, Department of Mathematics and Statistics, Winston-Salem, NC 27109, USA}
\email{bourdoam@wfu.edu}
\urladdr{http://users.wfu.edu/bourdoam/}

\author{\"Ozlem Ejder}
\address{Colorado State University, Department of Mathematics, Fort Collins, CO 80524, USA}
\email{ejder@math.colostate.edu}
\urladdr{https://sites.google.com/site/ozheidi/Home}

\author{Yuan Liu}
\address{University of Michigan, Department of Mathematics, Ann Arbor, MI 48109, USA}
\email{yyyliu@umich.edu}
\urladdr{https://sites.google.com/view/yuanliu-homepage}

\author{Frances Odumodu}
\address{Institut de Math{\'e}matiques de Bordeaux, Universit{\'e}  de Bordeaux, 
351 cours de la Lib{\'e}ration,
33405 Talence, cedex France; and \newline
Dipartimento di Matematica Pura ed Applicata, Universit\`a degli Studi di Padova,
Via Trieste 63, 35121, Padova, Italy}
\email{francesodumodu@gmail.com}

\author{Bianca Viray}
\address{University of Washington, Department of Mathematics, Box 354350, Seattle, WA 98195, USA}
\email{bviray@uw.edu}
\urladdr{http://math.washington.edu/~bviray}

\begin{document}

  \begin{abstract}
    We say a closed point $x$ on a curve $C$ is sporadic if $C$ has only finitely many closed points of degree at most $\deg(x)$ and {that $x$ is isolated if it is not in a family of effective degree $d$ {divisors} parametrized by $\PP^1$ or a positive rank abelian variety (see Section~\ref{sec:MapsOfCurves} for more precise definitions and a proof that sporadic points are isolated).}  Motivated by well-known classification problems concerning rational torsion of elliptic curves, we study sporadic {and isolated} points on the modular curves $X_1(N)$. In particular, we show that any non-cuspidal non-CM sporadic, {respectively isolated,} point $x \in X_1(N)$ maps down to a sporadic, {respectively isolated,} point on a modular curve $X_1(d)$, where $d$ is bounded by a constant depending only on $j(x)$. Conditionally, we show that $d$ is bounded by a constant depending only on the degree of $\Q(j(x))$, so in particular there are only finitely many $j$-invariants of bounded degree that give rise to sporadic {or isolated} points.  
	 \end{abstract}

	\maketitle
	
\section{Introduction}

Let $E$ be an elliptic curve over a number field $k$.  It is well-known that the torsion subgroup $E(k)_{\textup{tors}}$ is a finite subgroup of $(\Q/\Z)^2$.  In 1996, Merel \cite{Merel}, building on work of Mazur~\cite{Mazur} and Kamienny~\cite{Kamienny92}, proved the landmark uniform boundedness theorem: that for any positive integer $d$, there exists a constant $B = B(d)$ such that for all number fields $k$ of degree at most $d$ and all elliptic curves $E/k$, 
\[
  \#E(k)_{\textup{tors}} \leq B(d).
\]
Merel's theorem can equivalently be phrased as a statement about closed points on modular curves: that for any positive integer $d$, there exists a constant $B' = B'(d)$ such that for $n> B'$, the modular curve $X_1(n)/\Q$ has no non-cuspidal degree $d$ points.

Around the same time as Merel's work, Frey~\cite{Frey-InfinitelyManyDegreed} observed that Faltings's theorem implies that an arbitrary curve $C$ over a number field $k$ can have infinitely {many points of degree at most $d$} if and only if these infinitely many points are parametrized by $\PP^1_k$ or a positive rank subabelian variety of $\Jac(C)$.\footnote{{While Frey assumes that $C$ has a $k$-point, an inspection of the proof reveals that this is needed only to obtain a $k$-morphism $\Sym^d(C) \to \Jac(C)$.  Since the existence of a degree $d$ point also guarantees the existence of a suitable such morphism, the hypothesis on the existence of a rational point can be removed.}}  From this, Frey deduced that if a curve $C/k$ has infinitely many degree $d$-points, then the $k$-gonality of the curve\footnote{The $k$-gonality of a curve $C$ is the minimal degree of a $k$-rational map $\phi\colon C \to \PP^1_k$.} must be at most $2d$.  Frey's criterion combined with Abramovich's lower bound on the gonality of modular curves~\cite{Abramovich-Gonality} immediately shows that there exists a (computable!) constant $B'' = B''(d)$ such that for $n> B''$, the modular curve $X_1(n)/\Q$ has only finitely many degree $d$ points, {or in other words, that for $n> B''$ all degree $d$ points on $X_1(n)$ are \defi{sporadic}.\footnote{A closed point $x$ on a curve $C$ is sporadic if $C$ has only finitely many points of degree at most $\deg(x)$.}}  {Thus, the strength of the uniform boundedness theorem is in controlling the existence of sporadic points of bounded degree on $X_1(n)$ {as $n$ tends to infinity}.  }

In this paper, we study sporadic points {and, more generally, \defi{isolated}\footnote{A closed point $x$ on a curve $C$ is isolated if it is not contained in a family of effective degree $d$ divisors parametrized by $\PP^1$ or a positive rank abelian variety.  See Definition~\ref{def:Isolated} for more details.} points} of arbitrary degree, focusing particularly on such points corresponding to non-CM elliptic curves. We prove that non-CM non-cuspidal sporadic, {respectively isolated, }points on $X_1(n)$ map to sporadic, {respectively isolated,} points on $X_1(d)$, for $d$ some bounded divisor of $n$. 
\begin{theorem}\label{thm:UnconditionalMain}
  Fix a non-CM elliptic curve $E$ over $k$, and let $m$ be an integer divisible by $2,3$ and all primes $\ell$ where the $\ell$-adic Galois representation of $E$ is not surjective. Let $M = M(E, m)$ be the level of the $m$-adic Galois representation of $E$ and let $f$ denote the natural map $X_1(n) \to X_1(\gcd(n,M))$.
  If $x\in X_1(n)$ is  sporadic, {respectively isolated}, with $j(x) = j(E)$, then $f(x) \in X_1(\gcd(n,M))$ is  sporadic, {respectively isolated}.
\end{theorem}
{For many elliptic curves, we may take both $m$ and $M$ to be quite small.  For instance, let $\calE$ be the set of elliptic curves over $\Q$ where the $\ell$-adic Galois representation is surjective for all $\ell> 3$ and where the $6$-adic Galois representation has level dividing $24$.  Note that $\calE$ contains all Serre curves~\cite[Proof of Prop. 22]{Serre-OpenImage} (that is, elliptic curves over $\Q$ whose adelic image is of index 2 in $\GL_2(\Zhat)$, which is as large as possible) and hence contains almost all elliptic curves over $\Q$ when counted according to height~\cite[Theorem 4]{Jones10}.  For $E \in \calE$, we may apply Theorem~\ref{thm:UnconditionalMain} with $m = 6$ and $M|24$.

The curve $X_1(24)$ has infinitely many quartic points, but no rational or quadratic points, nor cubic points corresponding to elliptic curves over $\Q$~\cites{Mazur, KenkuMomoseQuadratic, Morrow}. Therefore $X_1(24)$ has no sporadic points with $\Q$-rational $j$-invariant.  For $M$ a proper divisor of $24$, the curves $X_1(M)$ have genus $0$, and so also have no sporadic points. Hence Theorem~\ref{thm:UnconditionalMain} yields the following corollary.
\begin{cor}
    For all $n$, there are no sporadic points on $X_1(n)$ corresponding to elliptic curves in $\calE$. In particular, there are no sporadic points corresponding to Serre curves.
\end{cor}}
{In addition to giving strong control on sporadic points over a fixed $j$-invariant, we are also able to use Theorem~\ref{thm:UnconditionalMain} to derive a uniform version that is 
conditional on a folklore conjecture motivated by a question of Serre.}
\begin{conj}[Uniformity Conjecture]\label{conj:Uniformity}
  Fix a number field $k$.  There exists a constant $C = C(k)$ such that for {all non-CM} elliptic curves $E/k$, the mod-$\ell$ Galois representation of $E$ is surjective for all $\ell > C$.
\end{conj}
\begin{conj}[Strong Uniformity Conjecture]\label{conj:StrongUniformity}
  Fix a positive integer $d$.  There exists a constant $C = C(d)$ such that for all degree $d$ number fields $k$ and all {non-CM} elliptic curves $E/k$, the mod-$\ell$ Galois representation of $E$ is surjective for all $\ell > C$.
\end{conj}
\begin{remark}\label{rmk:ConjForQ}
  {Conjecture~\ref{conj:Uniformity} when $k = \Q$, or equivalently Conjecture~\ref{conj:StrongUniformity} when $d=1$, is the case originally considered by Serre~\cite[\S4.3]{Serre-OpenImage}.  In this case, Serre asked whether $C$ could be taken to be $37$~\cite[p.399]{Serre-37}. The choice $C=37$ has been formally conjectured by Zywina~\cite[Conj. 1.12]{Zywina-PossibleImages} and Sutherland~\cite[Conj. 1.1]{Sutherland-ComputingGaloisImages}.}
\end{remark}

\begin{theorem}\label{thm:Main}
Assume Conjecture~\ref{conj:Uniformity}.  Then for any number field $k$, there exists a positive integer $M=M(k)$ such that if $x \in X_1(n)$ is a non-cuspidal, non-CM sporadic, {respectively isolated,} point with $j(x) \in k$, then $\pi(x) \in X_1(\gcd(n,M))$ is a sporadic, {respectively isolated,} point. Moreover, if the stronger Conjecture~\ref{conj:StrongUniformity} holds, then $M$ depends only on $[k:\Q]$.
\end{theorem}
We call a point in $j\in X_1(1)\isom \PP^1$ an {isolated} $j$-invariant if it is the image of an {isolated} point on $X_1(n)$, for some positive integer $n$ . Since {any curve only has finitely many isolated points (see Theorem~\ref{thm:FiniteIsolated}\eqref{case:FinitelyManyIsolated})} and there are only finitely many CM $j$-invariants of bounded degree, we immediately obtain the following corollary.
\begin{cor}Fix a number field $k$.\hfill
  \begin{enumerate}[leftmargin=*]
    \item[a)]  Assume Conjecture~\ref{conj:Uniformity}. There are finitely many $k$-rational isolated $j$-invariants. 
    \item[b)] Assume Conjecture \ref{conj:StrongUniformity}. There are finitely many isolated $j$-invariants of bounded degree.
  \end{enumerate}
\end{cor}
The integer $M$ in Theorem \ref{thm:Main} depends on the constant $C(k)$ or $C(d)$ from Conjecture~\ref{conj:Uniformity} or Conjecture~\ref{conj:StrongUniformity}, respectively, and also depends on a uniform bound for the level of the $\ell$-adic Galois representation  
for all $\ell\leq C(k)$, respectively $C(d)$.  The existence of this latter bound depends on Faltings's Theorem and as such is ineffective.  However, in the case when $k = \Q$, it is possible to make a reasonable guess for $M$.  This is discussed more in Section~\ref{sec:RationaljInvariant}.

\subsection{Prior work}\label{subsec:PriorWork}
CM elliptic curves provide a natural class of examples of sporadic points due to fundamental constraints on the image of the associated Galois representation. Indeed, Clark, Cook, Rice, and Stankewicz show that there exist sporadic points corresponding to CM elliptic curves on $X_1(\ell)$ for all sufficiently large primes $\ell$~\cite{CCS}.  Sutherland has extended this argument to include composite integers~\cite{Sutherland-Survey}.

In the non-CM case, all known results on sporadic points have arisen from explicit versions of Merel's theorem for low degree.  For instance, in studying cubic points on $\cup_{n\in \NN}X_1(n)$, Najman identified two degree $3$ sporadic points on $X_1(21)$ all corresponding to the same non-CM elliptic curve with rational $j$-invariant~\cite{Najman}. 
Derickx, Etropolski, van Hoeij, Morrow, and Zureick-Brown {are currently classifying all degree $3$ non-cuspidal non-CM sporadic points on $X_1(n)$, and preliminary results suggest that these examples of Najman are the only examples} \cite{DEvHZB}. Work of van Hoeij~\cite{vanHoeij}, Derickx--van Hoeij \cite{DerickxVanHoeij}, and Derickx--Sutherland~\cite{DerickxSutherland} show that there are additional sporadic points, e.g., degree 5 points on $X_1(28)$ and $X_1(30)$ and a degree $6$ point on $X_1(37)$.

{There are also examples of isolated points that are not sporadic. Derickx and van Hoeij \cite{DerickxVanHoeij} have shown $X_1(25)$ has a (nonempty) finite collection of points of degree $d=6$ and $d=7$. These points are isolated by Theorem \ref{thm:FiniteIsolated}, but not sporadic since the $\Q$-gonality of $X_1(25)$ is 5. In general, having infinitely points on a curve of degree $d$ does not preclude the existence of isolated points of degree $d$: see \cites{Siksek, Bruin-Najman, Gunther-Morrow, Box} for some examples}.

\subsection{Outline}
We set notation and review relevant background in Section~\ref{sec:Background}. In Section~\ref{sec:AdelicGroups} we record results about subgroups of $\GL_2(\Zhat)$ that will be useful in later proofs; in particular, Proposition~\ref{prop:levelforcomp} is useful for determining the level of an $m$-adic Galois representation from information about the $\ell$-adic representations.  {In Section~\ref{sec:MapsOfCurves}, we prove a general criterion for the images of sporadic or isolated points to remain sporadic or isolated (Theorem~\ref{thm:PushingForwardSporadicIsolated}); this result is likely of independent interest.}  In Section~\ref{sec:FixedjInvariant}, we study isolated points on modular curves over a fixed non-CM $j$-invariant and prove Theorem~\ref{thm:UnconditionalMain}. This is then used in Section~\ref{sec:Uniform} to prove Theorem~\ref{thm:Main}. 

Theorem~\ref{thm:Main} implies that, assuming Conjecture~\ref{conj:StrongUniformity}, there are finitely many isolated $j$-invariants of bounded degree.  This raises two interesting questions:
  \begin{enumerate}
    \item Are there finitely many isolated \emph{points} lying over $j$-invariants of bounded degree, or can there be infinitely many isolated points over a single $j$-invariant? \label{question:LiftingSporadic}
    \item In the case of degree $1$, when there is strong evidence for Conjecture~\ref{conj:StrongUniformity}, can we come up with a candidate list for the rational isolated $j$-invariants?\label{question:RationalSporadic}
  \end{enumerate}
Question~\ref{question:LiftingSporadic} is the focus of Section~\ref{sec:LiftingSporadic}, where we show that any CM $j$-invariant has infinitely many isolated (and in fact, sporadic!) points lying over it.  Section~\ref{sec:RationaljInvariant} focuses on Question~\ref{question:RationalSporadic}; there we provide a candidate list of levels from which the rational isolated $j$-invariants can be found.

\section*{Acknowledgements}
  This project was started at the Women in Numbers 4 conference, which was held at the Banff International Research Station.  We thank BIRS for the excellent working conditions and the organizers, Jennifer Balakrishnan, Chantal David, Michelle Manes, and the last author, for their support.  We also thank the other funders of the conference: the Clay Mathematics Institute, Microsoft Research, the National Science Foundation, the Number Theory Foundation, and the Pacific Institute for Mathematics Sciences.  

  We thank Jeremy Rouse, Drew Sutherland, and David Zureick-Brown for helpful conversations.   {We also thank Nigel Boston, Pete L. Clark, Lo\"{i}c Merel, Filip Najman, Paul Pollack, and the anonymous referees for comments on an earlier draft.  }The third author was partially supported by NSF grant DMS-1652116 and DMS-1301690 and the last author was partially supported by NSF CAREER grant DMS-1553459.

\section{Background and notation}\label{sec:Background}

    \subsection{Conventions}
        Throughout, $k$ denotes a number field, $\Qbar$ denotes a fixed algebraic closure of $\Q$, and $\Gal_k$ denotes the absolute Galois group $\Gal(\Qbar/k)$.

        We use $\ell$ to denote a prime number and $\Z_{\ell}$ to denote the $\ell$-adic integers.  For any positive integer $m$, we write $\Supp(m)$ for the set of prime divisors of $m$ and write $\Z_m := \prod_{\ell\in\Supp(m)}\Z_{\ell}$.   We use $S$ to denote a set of primes, typically finite; when $S$ is finite, we write $\frakm_S := \prod_{\ell\in S}\ell.$

        For any subgroup $G$ of $\GL_2(\Zhat)$ and any positive integer $n$, we write $G_n$ and $G_{n^{\infty}}$, respectively for the images of $G$ under the projections 
        \[
            \GL_2(\Zhat)\to \GL_2(\Z/n\Z)
            \quad\textup{and}\quad
            \GL_2(\Zhat) \to \GL_2(\Z_n).
        \]
        In addition, for any positive integer $m$ relatively prime to $n$ we write $G_{n\cdot m^{\infty}}$ for the image of $G$ under the projection
        \[
            \GL_2(\Zhat) \to \GL_2(\Z/n\Z)\times \GL_2(\Z_m).   
        \]
        Throughout, we will abuse notation and use $\pi$ to denote any natural projection map among the groups $G, G_{n^{\infty}},$ and $G_n$.
        
        {By curve we mean a projective nonsingular geometrically integral $1$-dimensional scheme over a field.  For a curve $C$, we write $\kk(C)$ for the function field of $C$ and $\Pic_C$ for the Picard scheme of $C$.  For any non-negative integer $d$, we write $\Pic^d_C$ for the connected component of $\Pic_C$ consisting of divisor classes of degree $d$ and $\Sym^d C$ for the $d^{th}$ symmetric product of $C$, i.e., $C^d/S_d$.}  If $C$ is defined over the field $K$, we use $\gon_K(C)$ to denote the $K$-gonality of $C$, which is the minimum degree of a dominant morphism $C \rightarrow \mathbb{P}^1_K$. If $x$ is a closed point of $C$, we denote the residue field of $x$ by $\kk(x)$ and define the degree of $x$ to be the degree of the residue field $\kk(x)$ over $K$. A point $x$ on a curve $C/K$ is \defi{sporadic} if there are only finitely many points $y\in C$ with $\deg(y)\leq \deg(x)$.  {We also consider other related properties of a closed point on a curve: \defi{isolated}, $\defPisolated$, and $\defAVisolated$; these terms are defined in Section~\ref{sec:MapsOfCurves}.}
        
        We use $E$ to denote an elliptic curve, i.e., a curve of genus $1$ with a specified point $O$. {Unless stated otherwise,} we will consider only elliptic curves defined over {number fields}. We say that an elliptic curve $E$ over a field $K$ has complex multiplication, or {CM}, if the geometric endomorphism ring is strictly larger than $\mathbb{Z}$.  Given an elliptic curve $E$ over a number field $k$,  an affine model of $E$ is given by a short Weierstrasss equation $y^2=x^3+Ax+B$ for some $A,B \in k$.  Then the $j$-invariant of $E$ is  $j(E):=1728\frac{4A^3}{4A^3+27B^2}$ and uniquely determines the geometric isomorphism class of $E$.  For a positive integer $n$, we write $E[n]$ for the subgroup of $E$ consisting of points of order at most $n$.

    \subsection{Modular Curves} 
        For a positive integer $n$, let
        \[
            \Gamma_1(n) \coloneqq
            \left\{{\left(\begin{smallmatrix}a & b \\ c & d \end{smallmatrix}\right)} \in \SL_2(\Z) :  c \equiv 0 \pmod{n}, \, a \equiv d \equiv 1 \pmod{n}\right\}. 
        \]
        The group $\Gamma_1(n)$ acts on the upper half plane $\mathbb{H}$ via linear fractional transformations, and the points of the Riemann surface 
        \[
            Y_1(n) \coloneqq \mathbb{H}/\Gamma_1(n)
        \]
        correspond to $\mathbb{C}$-isomorphism classes of elliptic curves with a distinguished point of order $n$. That is, a point in $Y_1(n)$ corresponds to an equivalence class of pairs $[(E,P)]$, where $E$ is an elliptic curve over $\C$ and $P\in E$ is a point of order $n$, and where
        $(E,P) \sim (E',P')$ if there exists an isomorphism $\varphi\colon E \rightarrow E'$ such that $\varphi(P)=P'$. By adjoining a finite number of cusps to $Y_1(n)$, we obtain the smooth projective curve $X_1(n)$. {Concretely, we may define the extended upper half plane $\mathbb{H}^* \coloneqq \mathbb{H}\cup \Q \cup \{\infty\}$. Then $X_1(n)$ corresponds to the extended quotient $\mathbb{H}^*/\Gamma_1(n)$.}  In fact, we may view $X_1(n)$ as an algebraic curve defined over $\Q$ (see \cite[Section 7.7]{DiamondShurman} or~\cite{DeligneRapoport} for more details).

        \subsubsection{Degrees of non-cuspidal algebraic points}
        If $x=[(E,P)]\in X_1(n)(\Qbar)$ is a non-cuspidal point, then the moduli definition implies that
        $\deg(x)=[\Q(j(E), \mathfrak{h}(P)):\Q],$
        where $\mathfrak{h}: E \rightarrow E/\Aut(E) \cong \mathbb{P}^1$ is a Weber function for $E$. 
        From this we deduce the following lemma:
        \begin{lemma}\label{lem:degree}
            Let $E$ be a non-CM elliptic curve defined over the number field $k=\Q(j(E))$, let $P\in E$ be a point of order $n$, and let $x = [(E,P)]\in X_1(n)$. Then
            \[
                \deg(x)=c_x[k(P):\Q],
            \]
            where $c_x=1/2$ if {$2P \neq O$ and} there exists $\sigma \in \Gal_k$ such that {$\sigma(P)=-P$} and $c_x=1$ otherwise.
        \end{lemma}
        \begin{proof}
            Let $E$ be a non-CM elliptic curve defined over $k=\Q(j(E))$ and let $\mathfrak{h}$ be a Weber function for $E$. If $\sigma \in \Gal_{k(\mathfrak{h}(P))}$, then $\sigma(P)=\xi(P)$ for some $\xi \in \Aut(E)$.
            Thus in the case where $\Aut(E)=\{\pm 1\}$, 
            \[
                [k(P):k( \mathfrak{h}(P))]=1 \text{ or } 2.
            \]
            If there exists $\sigma \in \Gal_k$ such that $\sigma(P)=-P$, then $[k(P):k( \mathfrak{h}(P))]=2$ and $c_x=1/2$. Otherwise $k(P)=k( \mathfrak{h}(P))$ and $c_x=1$.
        \end{proof}

        \subsubsection{Maps between modular curves}
            \begin{proposition}\label{prop:Degree}
                For positive integers $a$ and $b$, there is a natural $\Q$-rational map $f\colon X_1(ab) \rightarrow X_1(a)$ with
                \[
                    \deg(f)=
                    c_{f}\cdot b^2 \prod_{p \mid b,\, p \nmid a}
                    \left(1-\frac{1}{p^2}\right),
                \]
                where $c_{f}=1/2$ if $a \leq 2$ and $ab>2$ and $c_{f}=1$ otherwise. 
            \end{proposition}
            \begin{proof} 
                Since $\Gamma_1(ab) \subset \Gamma_1(a)$, we have a natural  map $X_1(ab) \rightarrow X_1(a)$ that complex analytically is induced by $\Gamma_1(ab)\tau \mapsto \Gamma_1(a)\tau$ for {$\tau \in \mathbb{H}^*$}. On non-cuspidal points, this map has the moduli interpretation $[(E,P)] \mapsto [(E,bP)]$, which shows that it is $\Q$-rational. Since $-I\in \Gamma_1(n)$ if and only if $n|2$, the degree computation then follows from the formula~\cite[p.66] {DiamondShurman}, which states
                \[
                    \deg(f)=
                    \begin{cases}
                        [\Gamma_1(a):\Gamma_1(ab)]/2 & \textup{if } -I \in \Gamma_1(a)\textup{ and } -I \not\in \Gamma_1(ab)\\
                        [\Gamma_1(a):\Gamma_1(ab)] & \textup{otherwise.}\hfill
                    \end{cases}\qedhere
                \]
            \end{proof}

    \subsection{Galois Representations of Elliptic Curves}\label{sec:Serre}
        Let $E$ be an elliptic curve over a number field $k$. Let $n$ be a positive integer.  After fixing two generators for $E(\kbar)[n]$, we obtain a Galois representation
        \[
            \rho_{E, n} \colon \Gal_k \to \GL_2(\Z/n\Z).
        \]
        {Note that the conjugacy class of the image of $\rho_{E, n}$ is independent of the choice of generators.}
        After choosing compatible generators for each $n$, we obtain a Galois representation
        \[
            \rho_E \colon \Gal_k \to \GL_2(\Zhat) \isom \prod_{\ell}\GL_2(\Z_{\ell}),
        \]
        which agrees with $\rho_{E,n}$ after reduction modulo $n$.  For any positive integer $n$ we also define
        \[
            \rho_{E, n^{\infty}} \colon \Gal_k \to \GL_2(\Z_{n})
        \]
        to be the composition of $\rho_E$ with the projection onto the $\ell$-adic factors for $\ell|n$.  Note that $\rho_{E,n^{\infty}}$ depends only on the support of $n$.  We refer to $\rho_{E, n}$, $\rho_{E, n^{\infty}},$ and $\rho_{E}$ as the mod $n$ Galois representation of $E$, the $n$-adic Galois representation of $E$, and the adelic Galois representation of $E$, respectively.
        
        If $E{/k}$ does not have complex multiplication, then Serre's Open Image Theorem \cite{Serre-OpenImage} states that $\rho_{E}(\Gal_k)$ is open---and hence of finite index---in $\GL_2(\Zhat)$. Since the kernels of the natural projection maps $\GL_2(\Zhat) \rightarrow \GL_2(\Z/n\Z)$ form a fundamental system of open neighborhoods of the identity in $\GL_2(\Zhat)$~\cite[Lemma 2.1.1]{RibesZalesskii}, it follows that for any open subgroup $G$ of $\GL_2(\Zhat)$ there exists  $m\in \Z^+$ such that $G= \pi^{-1} (G \mod m)$. Thus Serre's Open Image Theorem can be rephrased in the following way: for any non-CM elliptic curve $E/k$, there exists a positive integer {$M$} such that
        \[
            \im \rho_{E} = \pi^{-1}(\im \rho_{E, M}).  
        \]
       {We call the smallest such $M$ the level and denote it $M_E$.}  {Similarly, for any finite set of primes $S$, we let $M_E(S)$ be the least positive integer such that $\im \rho_{E, \mm_S^\infty} = \pi^{-1}(\im\rho_{E, M_E(S)})$ and we say that $M_E(S)$ is the level of the $\mm_S$-adic Galois representation. }

        We also define
        {\begin{equation}\label{eq:DefnOfSE}
            S_E = S_{E/k} := \left\{2, 3\right\} \cup \left\{ \ell : \rho_{E, \ell^{\infty}}(\Gal_k) \not\supset \SL_2(\Z_{\ell}) \right\} \cup \left\{5, \textup{ if }\rho_{E, 5^{\infty}}(\Gal_k) \neq  \GL_2(\Z_5)\right\};
        \end{equation}}
        by Serre's Open Image Theorem, this is a finite set.

        For any elliptic curve $E/\Q$ with discriminant $\Delta_E$,\footnote{{While the discriminant depends on a Weierstrass model, the class of $\Delta_E\in \Q^{\times}/\Q^{\times2}$ is independent of the choice of model.  Since we are concerned only with $\Delta_E$ mod squares, we allow ourselves this abuse of notation.}} Serre observed that the field $\Q(\sqrt{\Delta_E})$ is contained in the $2$-division field $\Q(E[2])$ as well as a cyclotomic field $\Q(\mu_n)$ for some $n$, which in turn is contained in the $n$-division field $\Q(E[n])$.  Thus if $\ell>2$ is a prime that divides the squarefree part of $\Delta_E$, then $2\ell$ must divide the level $M_E$ {(see~\cite[Proof of Prop. 22]{Serre-OpenImage} for more details)}. In particular, the level of an elliptic curve can be arbitrarily large. In contrast, for a fixed prime $\ell$, the level of the $\ell$-adic Galois representation is bounded depending only on the degree of the field of definition. 
        \begin{theorem}[{\cite[Theorem 1.1]{CadoretTamagawa}, see also \cite[Theorem 2.3]{ClarkPollack}}]\label{thm:UniformIndex}
            Let $d$ be a positive integer and let $\ell$ be a prime number. There exists a constant {$C=C(d,\ell)$} such that for all number fields $k$ of degree $d$ and all non-CM elliptic curves $E{/k}$,
            \[
                [\GL_2(\Z_{\ell}):\im\rho_{E,\ell^{\infty}}]<C.
            \]
        \end{theorem}

\section{Subgroups of $\GL_2(\Zhat)$}\label{sec:AdelicGroups}

    The proofs in this paper involve a detailed study of the mod-$n$, $\ell$-adic and adelic Galois representations associated to elliptic curves.  As such, we use a number of properties of closed subgroups of $\GL_2(\Zhat)$ and subgroups of $\GL(\Z/n\Z)$ that we record here.  Throughout $G$ denotes a subgroup of $\GL_2(\Zhat)$.

    In Section~\ref{sec:Goursat}, we state Goursat's lemma.  In Section~\ref{sec:Kernels} we show that {if $\ell=5$ and $G_5 = \GL_2(\Z/5\Z)$ or if $\ell>5$ is a prime such that $G_{\ell} \supset \SL_2(\Z/\ell\Z)$}, then for any integer $n$ relatively prime to $\ell$, the kernel of the projection $G_{\ell^sn}\to G_n$ is large, in particular, it contains $\SL_2(\Z/\ell^s\Z)$.  This proof relies on a classification of subquotients of $\GL_2(\Z/n\Z)$: that $\GL_2(\Z/n\Z)$ can contain a subquotient isomorphic to {$\PGL_2(\Z/5\Z)$ or $\PSL_2(\Z/\ell\Z)$ for $\ell>5$ only if $5 | n$ or $\ell | n$ respectively}.  This result is known in the case $\ell> 5$ (see \cite[Appendix, Corollary 11]{CojocaruKani}), but we are not aware of a reference in the case $\ell = 5$.  In Section~\ref{sec:LevelIndex} we review results of Lang and Trotter that show that the level of a finite index subgroup of $\Gl_2(\Z_{\ell})$ can be bounded by its index.  Finally in Section~\ref{sec:madiclevel} we show how to obtain the $m$-adic level of a group from information of its $\ell$-adic components.

    \subsection{Goursat's Lemma}\label{sec:Goursat}
    \begin{lemma}[{Goursat's Lemma, see e.g., \cite[pg75]{Lang-algebra} or \cite{Goursat89}}]\label{lem:Goursat}
        Let $G, G'$ be groups and let $H$ be a subgroup of $G \times G'$ such that the two projection maps 
        \[
            \rho \colon H \to G \quad \textup{ and } \quad \rho' \colon H \to G'
        \] 
        are surjective. Let $N \coloneqq \ker(\rho)$ and $N' \coloneqq \ker (\rho')$; one can identify $N$ as a normal subgroup of $G'$ and $N'$ as a normal subgroup of $G$. Then the image of $H$ in $G/N' \times G'/N$ is the graph of an isomorphism
        \[ 
            G/N' \simeq G'/N.
        \]
    \end{lemma}

    \subsection{Kernels of reduction maps}\label{sec:Kernels}
    \begin{prop}\label{prop:surjectiveSL}
        {Let $\ell \geq 5$ be a prime.  Assume that $G_{\ell} \supset \SL_2(\Z/\ell\Z)$ when $\ell>5$ and $G_{\ell} = \GL_2(\Z/\ell\Z)$ when $\ell=5$. } Then 
        $\SL_2(\Z/\ell^s \Z) \subset \ker (G_{\ell^s n} \to G_n)$ for any positive integer $n$ with $\ell\nmid n$.
    \end{prop} 

    {For $\ell>5$, a key ingredient in the proof is a classification result that implies that $\PSL_2(\Z/\ell\Z)$ cannot appear as a subquotient of $G_n$ \cite[Appendix, Corollary 11]{CojocaruKani}.  This is false when $\ell = 5$ (for instance, there is a subquotient of $\GL_2(\Z/11\Z)$ that is isomorphic to $\PSL_2(\Z/5\Z)$).  However, we prove that $\PGL_2(\Z/5\Z)$ cannot be isomorphic to a subquotient of $G_n$ unless $5|n$.}

    \begin{lemma}\label{lem:Subquo}
        {Let $n$ be a positive integer.} If $\GL_2(\Z/n\Z)$ has a subquotient that is isomorphic to {$\PGL_2(\Z/5\Z)$}, then {$5\mid n$}.
	\end{lemma} 
    \begin{remark}
        Throughout the proof, we freely use the isomorphism $\PGL_2(\Z/5\Z)\isom S_5$ and $\PSL_2(\Z/5\Z)\isom A_5$ to deduce information about subgroups and subquotients contained in these groups.
    \end{remark}
    \begin{proof}
        The lemma is a straightforward consequence of the following $3$ claims (Claim~\eqref{Claim:InductionPGL} is applied to the set $T = \Supp(n)$).
        \begin{enumerate}
            \item The projection
            \[
                \GL_2(\Z/n\Z) \to \prod_{{p \in\Supp(n)}}\PGL_2(\Z/p\Z) 
            \]
            is an injection when restricted to any subquotient isomorphic to $\PGL_2(\Z/5\Z)$.\label{claim:ProjectiontoProdPGL}
            \item Let $\emptyset\neq S\subsetneq T$ be finite sets of primes.  If $\prod_{p\in T}\PGL_2(\Z/p\Z)$ has a subquotient isomorphic to $\PGL_2(\Z/5\Z)$ then so does at least one of
            \[
                \prod_{p\in S}\PGL_2(\Z/p\Z) \quad\textup{or}\quad
                \prod_{p\in T-S}\PGL_2(\Z/p\Z).
            \]
            Hence, by induction, if $\prod_{p\in T}\PGL_2(\Z/p\Z)$ has a subquotient isomorphic to $\PGL_2(\Z/5\Z)$ then $\PGL_2(\Z/p\Z)$ has a subquotient isomorphic to $\PGL_2(\Z/5\Z)$ for some $p\in T$.\label{Claim:InductionPGL}
            \item If $p$ is a prime and $\PGL_2(\Z/p\Z)$ has a subquotient isomorphic to $\PGL_2(\Z/5\Z)$, then $p=5$.\label{Claim:IsomorphicPGLs}
        \end{enumerate}
        
        \textbf{Proof of Claim~\ref{claim:ProjectiontoProdPGL}:} Let $N\vartriangleleft G < \GL_2(\Z/n\Z)$ be subgroups and let $\pi$ denote the surjective map
        \[
            \pi\colon\GL_2(\Z/n\Z) \to \prod_{p \in \Supp(n)}\PGL_2(\Z/p\Z).  
        \]
        Using the isomorphism theorems, we obtain the following
        \begin{equation}\label{eq:IsomorphismTheorems}
            \frac{\pi(G)}{\pi(N)}\isom \frac{G/(G\cap \ker \pi)}{N/(N\cap\ker \pi)}
            \isom\frac{G}{N\cdot(G\cap\ker\pi)}\isom
            \frac{G/N}{(G\cap \ker\pi)/(N\cap \ker\pi)}.
        \end{equation}
        For each prime $p$, the kernel of $\GL_2(\Z/{p}^m\Z)\to \GL_2(\Z/p\Z)$ is a $p$-group and the kernel of $\GL_2(\Z/p\Z)\to \PGL_2(\Z/p\Z)$ is a cyclic group, so $\ker \pi$ is a direct product of solvable groups. Hence $\ker \pi$ is solvable and so is $(G\cap\ker\pi)/(N\cap \ker \pi)$ for any $N \vartriangleleft G < \GL_2(\Z/n\Z)$. Since the only solvable normal subgroup of $\PGL_2(\Z/5\Z)$ is the trivial group, if $G/N \isom \PGL_2(\Z/5\Z)$, then $\pi(G)/\pi(N)\isom G/N$.

        \textbf{Proof of Claim~\ref{Claim:InductionPGL}:} Let $N\vartriangleleft G< \prod_{p\in T}\PGL_2(\Z/p\Z)$ be subgroups such that $G/N\isom \PGL_2(\Z/5\Z)$.  Let $H$ be the normal subgroup of $G$ containing $N$ such that $H/N\isom\PSL_2(\Z/5\Z)$.  Consider the following two maps
        \[
            \pi_S \colon \prod_{p\in T}\PGL_2(\Z/p\Z)\to \prod_{p\in S}\PGL_2(\Z/p\Z)
            \quad \textup{and}\quad
            \pi_{S^c} \colon \prod_{p\in T}\PGL_2(\Z/p\Z)\to \prod_{p\in T-S}\PGL_2(\Z/p\Z).  
        \]
        Since the only quotient of $\PGL_2(\Z/5\Z)$ that contains a subgroup isomorphic to $\PSL_2(\Z/5\Z)$ is $\PGL_2(\Z/5\Z)$ itself, by (\ref{eq:IsomorphismTheorems}) it suffices to show that either $\pi_S(H)/\pi_S(N)$ or $\pi_{S^c}(H)/\pi_{S^c}(N)$ is isomorphic to $\PSL_2(\Z/5\Z)$.  Furthermore, since $\PSL_2(\Z/5\Z)$ is simple, it suffices to rule out the case where $\pi_S(H)= \pi_S(N)$ and $\pi_{S^c}(H)= \pi_{S^c}(N)$, which by the isomorphism theorems are equivalent, respectively, to the conditions that
        \[
            \frac{H\cap \ker\pi_S}{N\cap \ker\pi_S}\isom \PSL_2(\Z/5\Z)
            \quad \textup{and} \quad
            \frac{H\cap \ker\pi_{S^c}}{N\cap \ker\pi_{S^c}}\isom \PSL_2(\Z/5\Z).
        \]
        Let
        \begin{align*}
            H_S := & (H\cap \ker \pi_S)\cdot(H\cap \ker\pi_{S^c}) \isom (H\cap \ker \pi_S)\times (H\cap \ker\pi_{S^c}),\\
            N_S := & (N\cap \ker \pi_S)\cdot(N\cap \ker\pi_{S^c}) \isom (N\cap \ker \pi_S)\times (N\cap \ker\pi_{S^c}).
        \end{align*}
        Assume by way of contradiction that $H_S/N_S\isom \frac{H\cap \ker\pi_S}{N\cap \ker\pi_S} \times \frac{H\cap \ker\pi_{S^c}}{N\cap \ker\pi_{S^c}} \isom \left(\PSL_2(\Z/5\Z)\right)^2$, and consider the normal subgroup $(H_S\cap N)/N_S$.  The isomorphism theorems yield an inclusion
        \[
            \frac{H_S/N_S}{(H_S\cap N)/N_S}\isom H_S/(H_S\cap N) \isom H_SN/N \hookrightarrow H/N \isom \PSL_2(\Z/5\Z),     
        \]
        so $(H_S\cap N)/N_S$ must be a nontrivial normal subgroup of $H_S/N_S$.  However, the only proper nontrivial normal subgroups of $\left(\PSL_2(\Z/5\Z)\right)^2$ are $\PSL_2(\Z/5\Z)\times \{1\}$ or $\{1\}\times \PSL_2(\Z/5\Z)$, so $N_S$ must contain either $H\cap \ker\pi_S$ or $H\cap \ker \pi_{S^c}$, which results in a contradiction.

        \textbf{Proof of Claim~\ref{Claim:IsomorphicPGLs}:}
        Let $G< \PGL_2(\Z/p\Z)$ be a subgroup that has a quotient isomorphic to $\PGL_2(\Z/5\Z)$.  If $p\nmid \#G$, then by~\cite[Section 2.5]{Serre-OpenImage}, $G$ must be isomorphic to a cyclic group, a dihedral group, $A_4$, $S_4$ or $A_5\cong \PSL_2(\Z/5\Z)$, so has no quotient isomorphic to $\PGL_2(\Z/5\Z)$. 
        Thus, $p$ must divide $\#G$.  Then $G\cap \PSL_2(\Z/p\Z)$ is also of order divisible by $p$ and so by \cite[Theorem~6.25, Chapter 3]{Suzuki}, $G\cap \PSL_2(\Z/p\Z)$ is solvable or equal to $\PSL_2(\Z/p\Z)$.  Since $G$ has a quotient isomorphic to $\PGL_2(\Z/5\Z)$, $G\cap \PSL_2(\Z/p\Z)$ cannot be solvable and hence $G = \PGL_2(\Z/p\Z)$ and $p=5$.
    \end{proof}       

    \begin{proof}[Proof of Proposition~\ref{prop:surjectiveSL}]
        Since we have $G_{\ell^sn}<\GL_2(\Z/\ell^sn\Z)\simeq \GL_2(\Z/\ell^s\Z) \times \GL_2(\Z/n\Z)$, there are natural surjective projection maps
		\[
			\pi_{s}: G_{\ell^s n} \to G_{\ell^s} \quad \text{and} \quad \varpi_s: G_{\ell^sn}\to G_n.
		\]
		Observe that $\ker \pi_{s}$ and $\ker \varpi_s$ can be identified as normal subgroups of $G_n$ and $G_{\ell^s}$ respectively, and by Goursat's Lemma (see Lemma~\ref{lem:Goursat}), we have 
            \begin{equation}\label{eq:Goursat}
                G_{\ell^s}/ \ker \varpi_s \isom G_n / \ker \pi_{s}.
            \end{equation}
		We first prove the proposition for the case when $s=1$.
		{By \cite[Theorem 4.9]{Artin57}, $\ker \varpi_1$ either contains $\SL_2(\Z/\ell\Z)$ or is a subgroup of the center $Z(\GL_2(\Z/\ell\Z))$ of $\GL_2(\Z/\ell\Z)$. If $\ker\varpi_1 \subseteq Z(\GL_2(\Z/\ell\Z))$ and $\ell=5$, then the left-hand side of \eqref{eq:Goursat} has a quotient $\PGL_2(\Z/5\Z)$, which contradicts Lemma~\ref{lem:Subquo} since the right-hand side cannot have such a quotient. Similarly, if $\ker \varpi_1 \subseteq Z(\GL_2(\Z/\ell{\Z}))$ and $\ell>5$, then the left-hand side of \eqref{eq:Goursat} has a subquotient $\PSL_2(\Z/\ell\Z)$, which is impossible by \cite[Appendix, Corollary~11]{CojocaruKani}. Therefore, $\ker \varpi_1$ must contain $\SL_2(\Z/\ell\Z)$.}
            
            For $s>1$, since $\varpi_s$ is surjective and factors through
            \[
            	G_{\ell^s n} \subset \GL_2(\Z/\ell^s n\Z) \to \GL_2(\Z/\ell n \Z) \to \GL_2(\Z/n\Z),
            \]
            $\ker \varpi_s \subset \GL_2(\Z/\ell^s \Z)$ maps surjectively onto $\ker\varpi_1\subset \GL_2(\Z/\ell\Z)$. Then the proposition follows from~\cite[Appendix, Lemma 12]{CojocaruKani}. 
	\end{proof}

    \subsection{Bounding the level from the index}\label{sec:LevelIndex}
        \begin{prop}[{\cite[Part I, \S6, Lemmas 2 \& 3]{LangTrotter-FrobeniusDistributions}}]
            \label{prop:MaximalKernelGivesLevel}
            Let $\ell$ be a prime and let $G$ be a closed subgroup of $\GL_2(\Z_{\ell})$.  Set $s_0 = 1$ if $\ell$ is odd and $s_0 = 2$ otherwise.  If 
            \[
                \ker (G \bmod \ell^{s+1} \to G \bmod \ell^s) = 
                I + \Mat_2(\ell^s\Z/\ell^{s+1}\Z)
            \]
            for some $s\geq s_0$,  then
            \[
                \ker (G \to G \bmod \ell^s) = I + \ell^s\Mat_2(\Z_{\ell}).
            \]
        \end{prop}
        \begin{remark}
            This proof follows the one given by Lang and Trotter.  We repeat it here for the reader's convenience and to show that the proof does give the lemma as stated, even though the statement of~\cite[Part I, \S6, Lemmas 2 \& 3]{LangTrotter-FrobeniusDistributions} is slightly weaker.
        \end{remark}
        \begin{proof}
            For any positive integer $n$, let $U_n := \ker (G \to G \bmod \ell^n)$ and let $V_n := I + \ell^n\Mat_2(\Z_{\ell})$.  Note that for all $n$, $U_n\subset V_n$ and $U_n = U_1 \cap V_n$.  

            Observe that for $s\geq s_0$, raising to the $\ell^{th}$ power gives the following maps
            \[
                V_s/V_{s+1} \stackrel{\sim}{\to} V_{s+1}/V_{s+2}, 
                \quad\textup{and}\quad
                U_s/U_{s+1} \hookrightarrow U_{s+1}/U_{s+2}.
            \]
            By assumption, the natural inclusion $U_s/U_{s+1} \subset V_s/V_{s+1}$ is an isomorphism for some $s\geq s_0$.  Combining these facts, we get the following commutative diagram for any positive $k$:
            \[
            \begin{tikzcd}        
                U_s/U_{s+1} \arrow[hook]{d} \arrow{r}{\sim} & V_s/V_{s+1}\arrow{d}{\sim}\\
                U_{s+k}/U_{s+k+1}   \arrow[hook]{r} & V_{s+k}/V_{s+k+1},
            \end{tikzcd}
            \]
            where the vertical maps are raising to the $({\ell^{k}})^{th}$ power and the horizontal maps are the natural inclusions.  Hence, $U_{s+k}/U_{s+k+1} = V_{s+k}/V_{s+k+1}$ for all $k\geq 0$ and so $U_s = V_s$.
        \end{proof}

    \subsection{Determining $m$-adic level from level of $\ell$-adic components}\label{sec:madiclevel}
            
        \begin{proposition}\label{prop:levelforcomp}
            Let $\ell_1, \dots, \ell_q$ be distinct primes and let $\mm := \prod_{i=1}^q \ell_i$. For $i = 1, \dots, q$, let $t_i\geq 1$ be positive integers and let $\mm_i := \prod_{j \neq i}{\ell_{j}}$.  If $G$ is a closed subgroup of $\GL_2(\Zhat)$ such that $G_{\mm_i\cdot{\ell_i}^{\infty}}=\pi^{-1}(G_{\mm_i{\ell_i}^{t_i}})$ for each $i$,
            then $G_{\mm^{\infty}}=\pi^{-1}(G_M)$ for $M=\prod_{i=1}^q {\ell_i}^{t_i}$.
        \end{proposition}
        \begin{proof}
            For any $1\leq i\leq q$ and $r_i\geq 0$, consider the following commutative diagram of natural reduction maps.
            \[
            \xymatrix{
                G_{M {\ell_i}^{r_i}} \ar@{->>}[r]\ar@{->>}[d] & G_{M} \ar@{->>}[d]\\
                G_{\mm_i {\ell_i}^{r_i+t_i}} \ar@{->>}[r] & G_{\mm_i {\ell_i}^{t_i}}}
            \]
            The kernel of the top horizontal map is a subgroup of $I + \Mat_2(M\Z/M\ell_i^{r_i}\Z)$, so its order is a power of $\ell_i$. Similarly, the order of the kernel of the lower horizontal map is a power of $\ell_i$, while the order of the kernels of the vertical maps are coprime to $\ell_i$. {Since $\# \ker(G_{M_i \ell_i^{r_i}} \to G_{M}) \cdot \# \ker(G_{M} \to G_{\mm_i\ell_i^{t_i}})$ is equal to $\# \ker(G_{M_i \ell_i^{r_i}} \to G_{\mm_i\ell_i^{r_i}}) \cdot \# \ker( G_{\mm_i \ell_i^{r_i}} \to G_{\mm_i\ell_i^{t_i}})$,} the kernels of horizontal maps must be isomorphic, and hence $G_{M{\ell_i}^{r_i}}$ is the full preimage of $G_{M}$, by assumption.

            To complete the proof, it remains to show that for any collection of positive integers $\{r_i\}_{i=1}^q$, $G_{M\prod_{i = 1}^q\ell_i^{r_i}}$ is the full preimage of $G_M$.  We do so with an inductive argument.  Let $1\leq q'\leq q$ and let $\{r_i\}_{i=1}^{q'}$ be a collection of positive integers. Consider the following commutative diagram of natural reduction maps.
            \[\xymatrix{
                G_{M\prod_{i=1}^{q'}{\ell_i}^{r_i}} \ar@{->>}[r]\ar@{->>}[d] & 
                G_{M\prod_{i=1}^{q'-1}{\ell_i}^{r_i}}\ar@{->>}[d]\\
                G_{M{\ell_{q'}}^{r_{q'}}}
                \ar@{->>}[r] & 
                G_{M}}                
            \]
            Again the kernels of the horizontal maps and the kernels of the vertical maps have coprime orders and so, by the induction hypothesis, the kernels of all maps are as large as possible.
        \end{proof}

  \section{Images of isolated points}
  \label{sec:MapsOfCurves}

    Let $C$ be a curve over a number field $F$ and consider the morphism
    \[
        \phi_d\colon \Sym^d C\to \Pic^d_C    
    \]
    that sends an unordered tuple of points to the sum of their divisor classes.  Let $W^d$ be the image of $\Sym^d C$ in $\Pic^d_C$. 
    Note that if there is a degree $d$ point on $C$ then $\Pic^d_C \isom \Pic^0_C$ and in particular is an abelian variety.

    \begin{definition}\label{def:Isolated}\hfill
        \begin{enumerate}
            \item A degree $d$ point $x\in C$ is $\defPisolated$ if there is no other point $x'\in (\Sym^d C)(F)$ such that $\phi_d(x) = \phi_d(x')$.  
            \item A degree $d$ point $x\in C$ is $\defAVisolated$ if there is no positive rank subabelian variety $A\subset \Pic^0_C$ such that $ \phi_d(x) + A \subset W^d$.
            \item A degree $d$ point $x\in C$ is \defi{isolated} if it is $\Pisolated$ and $\AVisolated$.
            \item A degree $d$ point $x\in C$ is \defi{sporadic} if there are only finitely many closed points $y\in C$ with {$\deg(y) \leq \deg(x)$}.
        \end{enumerate}
    \end{definition}
    Faltings's theorem \cite{Faltings} on rational points on subvarieties of abelian varieties implies the following two results on isolated and sporadic points.
    \begin{thm}\label{thm:FiniteIsolated}
        Let $C$ be a curve over a number field.
        \begin{enumerate}
            \item There are infinitely many degree $d$ points on $C$ if and only if there is a degree $d$ point on $C$ that is \emph{not} isolated.  In particular, sporadic points are isolated.\label{case:CharacterizingFiniteDegreed}
            \item There are only finitely many isolated points on $C$.\label{case:FinitelyManyIsolated}
        \end{enumerate}
    \end{thm}
    \noindent We provide the details of the proof in Section~\ref{sec:ProofTheoremFiniteIsolated}.

    In this section, we consider an arbitrary morphism of curves, and give a criterion for when images of isolated points remain isolated.  Our main result is the following.
  \begin{theorem}\label{thm:PushingForwardSporadicIsolated}
    Let $f\colon C \to D$ be a finite map of curves, let $x\in C$ be a closed point, and let $y = f(x) \in D$.  Assume that $\deg(x) = \deg(y)\cdot\deg(f)$.
    \begin{enumerate}
        \item If $x$ is $\Pisolated$, then $y$ is $\Pisolated$.\label{case:Pisolated}
        \item If $x$ is $\AVisolated$, then $y$ is $\AVisolated$.\label{case:Aisolated}
        \item If $x$ is sporadic, then $y$ is sporadic.\label{case:sporadic}
    \end{enumerate}
  \end{theorem}
  \begin{proof}
    Let $d = \deg(y)$ and let $e = \deg(f)$.  Then by assumption $de = \deg(x)$.

    \eqref{case:Pisolated} {Assume that $y$ is not $\Pisolated$, so there exists a point $y'\in (\Sym^dC)(F)$, different from $y$, such that $\phi_d(y) = \phi_d(y')$, or, in other words such that there exists a function $g\in \kk(D)^{\times}$ such that $\divv(g) = y - y'$. Since $y$ is a degree $d$ point (and not just an effective degree $d$ {divisor}), the assumption that $y\neq y'$ implies that $y$ and $y'$ have distinct support.}  Therefore the map $g\colon D \to \PP^1$ has degree $d$, and hence $g\circ f$ gives a degree $de$ map.  Then for any $z\in \PP^1(F)$ different from $g(f(x))$,  the fiber $(g\circ f)^{-1}(z)$ gives a point of $(\Sym^{{de}}C)(F)$, distinct from $x$, such that $\phi_{de}(x)=\phi_{de}((g\circ f)^{-1}(z)).$ In particular, $x$ is not $\Pisolated$.

    \eqref{case:Aisolated} Assume that $y$ is not $\AVisolated$, so there exists a subabelian variety  $A\subset \Pic^0_D$ such that $\phi_d(y)+ A\subset W^d$.  The morphism $f$ induces a commutative diagram
    \[
        \xymatrix{
            \Sym^d D \ar[r]^{\phi_d} \ar[d] & \Pic^d_D \ar[d]^{f^*}\\
            \Sym^{de} C \ar[r]^{\phi_{de}}  & \Pic^{de}_C,
        }
    \]
    where the left vertical arrow sends $y$ to $x$.  Therefore, $\phi_{de}(x) + f^*A\subset W^{de}$.  Since $f^*A$ is a positive rank subabelian variety of $\Pic^0 C$, the point $x$ is not $\AVisolated$.

    \eqref{case:sporadic} Assume that $y$ is not sporadic, i.e., that there are infinitely many closed points $y'\in D$ with $\deg(y')\leq \deg(y) = d$.  For each of these points $y'$, there is a closed point $x'\in f^{-1}(y')$ such that 
    \[
        \deg(x') \leq \deg(y')e \leq \deg(y)e = de = \deg(x).
    \]
    Hence, the point $x$ is not sporadic.
  \end{proof}

  \subsection{Proof of Theorem~\ref{thm:FiniteIsolated}}\label{sec:ProofTheoremFiniteIsolated}
    {\eqref{case:CharacterizingFiniteDegreed}} The forward direction is a straightforward consequence of Faltings's theorem  \cite{Faltings}; we include the details for the readers' convenience.  Assume that there are infinitely many degree $d$ points.  Then either there are two degree $d$ points $x, x'\in C$ such that $\phi_d(x) = \phi_d(x')$ and so in particular $x$ and $x'$ are not $\Pisolated$, or $\phi_d$ is injective on the set of degree $d$ points.  In the latter case,  $W^d\subset \Pic_C^d$ contains infinitely many rational points.  Faltings's theorem states that the rational points on $W^d$ are a finite union of translates of subabelian varieties, so in particular, there must be a positive rank abelian variety $A\subset \Pic^0_C$ and a degree $d$ point $x\in C$ such that $x + A\subset W^d$, i.e., the degree $d$ point $x$ is not $\AVisolated$. 

    Now we prove the backwards direction, which requires a more detailed study of Faltings's theorem.
    Let $x\in C$ be a degree $d$ point that is not isolated.  If $x$ is not $\Pisolated$, then there exists an $x'\in (\Sym^d C)(F), x\ne x'$, such that $\phi_d(x) = \phi_d(x')$, or equivalently, there exists a rational function $g\in \kk(C)^{\times}$ such that $\divv(g) = x - x'$.  Since $x$ is a closed point and $x'\in (\Sym^d C)(F)$, $x\ne x'$ implies that $x$ and $x'$ have disjoint support.  Thus the function $g$ gives a degree $d$ morphism $g\colon C \to \PP^1$.  By Hilbert's irreducibility theorem~\cite[Chap. 9]{Serre-MWLectures}, there are infinitely many degree $1$ points $z\in \PP^1$ such that $g^{-1}(z)$ has degree $d$, which gives the desired result.
    
    Now assume that $x$ is $\Pisolated$ but not $\AVisolated$, i.e., that $x$ is not equivalent to any other effective divisors and that there is a  positive rank subabelian variety $A\subset \Pic^0 C$ such that $x + T \subset W^d$.  Since the cokernel of $\Pic C \to \left(\Pic \overline{C}\right)^{\Gal(\overline{F}/F)}$ is torsion, there is a finite index subgroup $H\subset A(F)$ such that every divisor class in $H$ (and therefore every divisor class in $x + H$) is represented by an $F$-rational divisor, and every divisor class in $A(F)\setminus H$ is \emph{not} represented by an $F$-rational divisor.  In other words, $\phi_d\left((\Sym^d C)(F)\right)\cap (x + A(F)) = x + H$.
    
    Since $H$ has positive rank, taking the preimage of $x + H$ under $\phi_d$ yields infinitely many rational points on $\Sym^d C$, or, equivalently, infinitely many effective degree $d$ $0$-cycles on $C$.  It remains to prove that infinitely many of these $0$-cycles are irreducible, i.e., are not in the image of $\cup_i\left((\Sym^{d-i}C)(F)\times (\Sym^i C)(F)\right)$.
    
    Consider the following commutative diagram
    \[
        \xymatrix{
            (\Sym^{d-i} C)(F) \times (\Sym^i C)(F) \ar[rr]^(.63){\phi_{d-i}\times \phi_i} \ar[d] && W^{d-i}  \times W^i \ar[d]\\
            (\Sym^{d} C)(F) \ar[rr]^{\phi_{d}}  && W^{d},
        }
    \]
    where the vertical maps are induced by concatenation and summation, respectively.  If there are only finitely many degree $d$ points on $C$, then all but finitely many of the points in $x + H$ are contained in the union  $\cup_i (W^{d-i}(F) + W^i(F))$.  Faltings's theorem on rational points on subvarieties of abelian varieties implies that 
    \begin{equation}\label{eq:ImageOfLowerWd}
        \bigcup_{i=1}^{\lfloor d/2\rfloor} (W^{d-i}(F) + W^i(F)) = \bigcup_{j=1}^n y_j + A_j(F),
    \end{equation}
    where $n$ is some nonnegative integer, the $A_j$'s are some subabelian varieties of $\Pic^0_C$ and the $y_j$'s are degree $d$ divisors on $C$, which can be taken to be reducible and effective.
    
    We are concerned with the intersection
    \begin{align*}
        (x + H) \cap \left(\bigcup_{j=1}^n y_j + A_j(F)\right) & = 
        x + \left(H \cap \left(\bigcup_{j=1}^n y_j - x + A_j(F)\right)\right)\\
        &= x + \bigcup_{j=1}^n \left(H\cap \left(y_j-x + A_j(F)\right)\right)
    \end{align*}
    If the intersection $H\cap (y_j-x + A_j(F))$ is nonempty, then it is a coset of $H\cap A_j(F)$.  In addition, since $x$ is a $\Pisolated$ degree $d$ point on $C$, $x$ cannot be written as the sum of two nonzero effective divisors, so by definition of~\eqref{eq:ImageOfLowerWd}, $H\cap (y_j-x + A_j(F))$ does not include the identity.  Thus,
    \[
        (x + H) \cap \left(\bigcup_{j=1}^n y_j + A_j(F)\right)  = x + \bigcup_{j\in J} \left(z_j + H\cap A_j(F)\right),
    \] 
    where $J\subset \{1, \dots, n\}$ and $z_j\in H\setminus H\cap A_j(F)$.

    For each $j\in J$, let $G_j$ be a subgroup of $H$ of finite index that contains $H\cap A_j(F)$ and that does not contain $z_j$.  Then we have
    \begin{align*}
        (x + H) \cap \left(\bigcup_{j=1}^n y_j + A_j(F)\right) & = x + \bigcup_{j\in J} \left(z_j + H\cap A_j(F)\right)\\
        & \subset x + \bigcup_{j\in J} \left(z_j +G_j\right)\\
        & \subset (x+ H) \setminus \left(x + \cap_{j\in J}G_j\right)\subsetneq x + H.
    \end{align*}
    Since each $G_j$ is finite index in $H$, so is the intersection $\cap_{j\in J} G_j$.  Hence, the image of $\cup_i (W^{d-i}(F) + W^i(F))$ misses infinitely many rational points of $x+H$, and so there are infinitely many degree $d$ points on $C$.

    {\eqref{case:FinitelyManyIsolated}}       If $C$ has genus $0$, then no point is $\PP^1$-isolated, so the statement trivially holds.  Now assume that $g:= \textup{genus}(C)$ is positive,  and let $P$ be a point of degree $d$.  If $d> g + 1$, then by Riemann-Roch, $\ell(P) = d - g + 1 > 2$ and so $P$ is not $\PP^1$-isolated.   Therefore, any isolated point on $C$ must have bounded degree, and so it suffices to prove that there are only finitely many isolated points of a fixed degree $d$.

    Recall that degree $d$ points on $C$ give rise to rational points on $\Sym^d C$ that in turn map, via $\phi_d$, injectively to rational points on $W^d$. By Faltings's theorem, $W^d(F)$ is the finite union of translates of subabelian varieties of $\Pic^0 C$.  By definition, any degree $d$ point on $C$ that lands in a translate of a positive rank subabelian variety is not $\AVisolated$.  Therefore, the set of degree $d$ isolated points of $C$ must inject (under $\phi_d$) into a finite union of translates of rank $0$ subabelian varieties, so in particular must be finite.
\qed

  \section{Isolated points above a fixed non-CM $j$-invariant}
  \label{sec:FixedjInvariant}
  
    For any non-CM elliptic curve $E$ over a number field $k$, recall from \S\ref{sec:Serre} that {
    \[
        S_E = S_{E/k} := \left\{2, 3\right\} \cup \left\{\ell : \rho_{E, \ell^{\infty}}(\Gal_k) \not\supset  \SL_2(\Z_{\ell})\right\} \cup \left\{5, \textup{if }\rho_{E, 5^{\infty}}(\Gal_k) \neq  \GL_2(\Z_{5})\right\}.
    \]}

    {In this section we show that the degree of a non-cuspidal non-CM point $x \in X_1(n)$ is as large as possible given the degree of its image in $X_1(a)$ for $a = \gcd(n, M_{E_{x}}(S_{E_{x}}))$, where $E_{x}$ is an elliptic curve over $\Q(j(x))$ with $j$-invariant $j(x)$.  }
	\begin{theorem}\label{thm:DegreeOfMapFixedjInvariant}
        Fix a non-CM elliptic curve $E$ over a {number field $k$}.   Let $S$ be a finite set of places containing $S_E$ and let {$\mm_S := \prod_{\ell\in S}\ell$. Let} $M$ be a positive integer with $\Supp(M)\subset S$ satisfying
        \begin{equation}\label{eq:FullPreimage}
            \im \rho_{E, \mm_S^\infty} = \pi^{-1}(\im\rho_{E, M}).
        \end{equation}
          If $x\in X_1(n)$ is a closed point with $j(x) = j(E)$, then {
            $\deg(x) = \deg(f)\deg(f(x)),$
          where $f$ denotes the natural map $X_1(n) \to X_1(\gcd(n,M)).$}
    \end{theorem}
    \begin{remark}\label{Rmk:Leveljinvariant}
            Note that if $E$ and $E'$ are quadratic twists of each other, both defined over a {number field $k$}, then $S_E = S_{E'}$ (see, e.g., \cite[Lemma 5.27]{Sutherland-ComputingGaloisImages}).  Furthermore, for any $S$, 
            \begin{equation}\label{eq:MaximalDegreeGrowth}
                \pm (\im\rho_{E, \mm_{S}^{\infty}}) = \pm (\im\rho_{E', \mm_{S}^{\infty}})
            \end{equation}
            (see, e.g., \cite[Lemma 5.17]{Sutherland-ComputingGaloisImages}).  Since any open subgroup of $\GL_2(\Z_{\mm_S})$ has only finitely many subgroups of index $2$, there is an integer $M$ that will satisfy~\eqref{eq:FullPreimage} for all quadratic twists of a fixed elliptic curve. 
    \end{remark}
    {This theorem combined with Theorem~\ref{thm:PushingForwardSporadicIsolated} yields the following corollary, of which Theorem~\ref{thm:UnconditionalMain} is a special case.
    \begin{cor}\label{cor:SporadicPtsFixedjInvariant}
        Fix a non-CM elliptic curve $E$ over a number field $k$.   Let $S$ be a finite set of places containing $S_E$ and let $M$ be a positive integer with $\Supp(M)\subset S$ satisfying
        \[
            \im \rho_{E, \mm_S^\infty} = \pi^{-1}(\im\rho_{E, M}).
        \]
        Let $x\in X_1(n)$ be a point with $j(x) = j(E)$, and let $f$ denote the natural map $X_1(n) \to X_1(\gcd(n,M)).$
        \begin{enumerate}
            \item If $x$ is $\Pisolated$, then $f(x)$ is $\Pisolated$.
            \item If $x$ is $\AVisolated$, then $f(x)$ is $\AVisolated$.
            \item If $x$ is sporadic, then $f(x)$ is sporadic.
        \end{enumerate}
    \end{cor}
    From this, we deduce the following.}
    \begin{cor}\label{cor:SurjectivePrimes}
      Let $E$ be a non-CM elliptic curve defined over $k := \Q(j(E))$.  If $\ell\notin S_E$, then there are no sporadic {or isolated} points on $X_1(\ell^s)$ lying over $j(E)$ for any $s\in \NN$. 
    \end{cor}
    
    {In Section~\ref{subsec:SporadicPointsLowerLevel}, specifically Lemma~\ref{lem:LargeGaloisImageDegreeOfMap}, we show that the desired maximal degree growth condition (i.e., the conclusion of Theorem~\ref{thm:DegreeOfMapFixedjInvariant}) is implied by a condition on the degree of field extensions $k(P)/k(bP)$ where $P$ is a point of order $ab$ on a non-CM elliptic curve $E$.  
    We then show that the assumed growth of the Galois representation~\eqref{eq:FullPreimage} implies the hypothesis of Lemma~\ref{lem:LargeGaloisImageDegreeOfMap} in two different cases. First, for maps $X_1(n) \to X_1(n\ell^{-1})$ for prime divisors $\ell$ of $n$ outside of $S_E$ (see Section~\ref{subsec:SurjectivePrimes}), and second, for maps $X_1(ab)\to X_1(a)$ for integers $a,b$ with bounded support (see Section~\ref{subsec:FixedSupport}).  The results of these two sections are brought together in Section~\ref{subsec:ProofOfFixedjInvariant} to prove Theorem~\ref{thm:DegreeOfMapFixedjInvariant}.}

    \begin{remark}
        As discussed in Section~\ref{sec:Serre}, the full strength of Serre's Open Image Theorem implies that for any non-CM elliptic curve $E/k$, there exists a positive integer {$M_E$} such that
        \[
            \im \rho_{E} = \pi^{-1}(\im \rho_{E, {M_E}}).  
        \]
        The arguments in Section~\ref{subsec:FixedSupport} alone then imply that {$\deg(x) = \deg(f)\deg(f(x)),$
        where $f$ denotes the natural map $X_1(n) \to X_1(\gcd(n,M_E))$}, which yields a weaker version of Theorem~\ref{thm:DegreeOfMapFixedjInvariant}.  
        
        While there is not a dramatic difference in the strength of these results for a fixed elliptic curve, the difference is substantial when applied to a family of elliptic curves.  It is well-known that {$M_{E}$} can be arbitrarily large for a non-CM elliptic curve {$E$} over a fixed number field $k$ (see Section~\ref{sec:Serre}). However, for a fixed finite set of places $S$, we prove that $M_{E}(S)$ can be bounded depending only on $[k:\Q]$.  This allows us to obtain the uniform version of Corollary~\ref{cor:SporadicPtsFixedjInvariant}, namely Theorem~\ref{thm:Main} (see \S\ref{sec:Uniform}).
    \end{remark}
 
    \subsection{Field-theoretic condition for maximal degree growth}\label{subsec:SporadicPointsLowerLevel}

    \begin{lemma}\label{lem:LargeGaloisImageDegreeOfMap}
        Let $a$ and $b$ be positive integers, $E$ a non-CM elliptic curve over {a number field} $k$, and $P\in E$ a point of order $ab$. Let $x:=[(E,P)] \in X_1(ab)$ and let $f$ denote the map $X_1(ab) \to X_1(a)$.  If $[k(P) : k(bP)]$ is as large as possible, i.e., if $[k(P): k(bP)] =\# \{Q\in E: bQ = bP, Q \textup{ order }ab\}$, then 
        \[
            \deg(x) = {\deg(f)}\deg(f(x)).
        \]
    \end{lemma}
    \begin{proof}
        From the definition of $X_1(n)$, we have that
        \begin{equation}\label{eq:DegreeModularCurveMap}
            \#\{Q\in E : bQ = bP, Q \textup{ order }ab\} = 
            \begin{cases}
                2\deg(X_1(ab) \to X_1(a)) & \textup{if }a\leq2
                \textup{ and }ab>2,\\
                \phantom{2}\deg(X_1(ab) \to X_1(a)) & \textup{otherwise}.
            \end{cases}    
        \end{equation}
        Let us first consider the case that $a \leq 2$ and $ab > 2$.  Then $\deg(f(x)) = [k(bP):\Q]$.  Since $[k(P):k(bP)]$ is as large as possible and $a\leq 2$, there must be a $\sigma \in \Gal_k$ such that $\sigma(P) = -P$.  Hence $\deg(x) = \frac12[k(P):\Q]$ by Lemma \ref{lem:degree}, so~\eqref{eq:DegreeModularCurveMap} yields the desired result.

        Now assume that $ab \leq 2$.  Then $\deg(f(x)) = [k(bP):\Q]$ and $\deg(x) = [k(P):\Q]$, so~\eqref{eq:DegreeModularCurveMap} again yields the desired result. 

        Finally we consider the case when $a>2$.  Note that for any point $y\in X_1(ab)$, $\deg(y) \leq \deg(f(y)) \cdot \deg(X_1(ab)\to X_1(a))$.  Combining this with~\eqref{eq:DegreeModularCurveMap}, it remains to prove that
		\[
			\frac{\deg(x)}{\deg(f(x))} \geq \#\{Q \in E : bQ = bP, Q \textup{ order }ab\}.
		\]
		{By Lemma \ref{lem:degree}, $\deg(x) = c_x \cdot [k(P): \Q]$ and $\deg(f(x)) = c_{f(x)}\cdot [k(bP):\Q]$ where $c_x, c_{f(x)} \in \{1, 1/2\}.$  
		Since any $\sigma\in \Gal_k$ that sends $P$ to $-P$ also sends $bP$ to $-bP$}, $c_x\geq c_{f(x)}$ and so these arguments together show that
		\[
			\frac{\deg(x)}{\deg(f(x))} = \frac{c_x  [k(P): \Q]}{c_{f(x)} [k(bP):\Q]} = \frac{c_x}{c_{f(x)}} [k(P):k(bP)] \geq [k(P):k(bP)].
		\] 
		By assumption, $[k(P):k(bP)] = \#\{Q \in E : bQ = bP, Q \textup{ order }ab\}$, yielding the desired inequality.
    \end{proof}

    \subsection{Eliminating primes with {large} Galois representation}\label{subsec:SurjectivePrimes}
    \begin{prop}\label{prop:EliminatingSurjectivePrimes}    
        Let $E$ be a non-CM elliptic curve over {a number field $k$}, let $\ell$ be a prime {not contained in $S_E$}, and let $a$ and $s$ {be positive integers}.  Let $x\in X_1(a\ell^s)$ {be a closed point with $j(x) = j(E)$ and let $f\colon X_1(a\ell^s) \to X_1(a)$ be the natural map.  Then
        \[
            \deg(x) = \deg(f)\deg(f(x)).    
        \]}
    \end{prop}        
    \begin{proof}
        {Write $a = b\ell^t$ where $\ell\nmid b$, let $g$ denote the map $X_1(a) \to X_1(b)$ and let $h\colon X_1(a\ell^s) \to X_1(b)$ be the composition $g\circ f$.  Since $\deg(h) = \deg(f)\deg(g)$, the general case follows from the case when $\ell\nmid a$.  We work with this assumption for the remainder of the proof.}

        {Let $P\in E$ be a point of order $a\ell^s$ such that $x = [(E,P)]$ and for any $c|a\ell^s$, let $B_{c}^1\subset \Aut(E[c])$ be the stabilizer of $\frac{a\ell^s}{c}P$.}  Let $H$ denote the kernel of the projection map $\im \rho_{E, a\ell^s} \to \im \rho_{E, a}$.

        {We wish to prove $[k(P) : k(\ell^sP)] =  \#\{ Q\in E: \ell^sQ=\ell^sP,\; Q \text{ order } a\ell^s\}$, so that we can apply Lemma~\ref{lem:LargeGaloisImageDegreeOfMap}.}  Note that we always have the following upper bound
        \[
            \#\left(\Aut(E[\ell^s])/B_{\ell^s}^1\right) = \#\{ Q\in E: \ell^sQ=\ell^sP,\; Q \text{ order } a\ell^s\} \geq [k(P):k(\ell^sP)].
        \]
        We may also apply Galois theory to the towers of fields $k(E[a\ell^s])\supset k(P) \supset k(\ell^sP)$ and $k(E[a\ell^s])\supset k(E[a]) \supset k(\ell^sP)$ to obtain the following lower bound.
        \begin{align*}
            [k(P):k(\ell^s P)] & = 
            \frac{[k(E[a\ell^s]):k(E[a])]\cdot
            [k(E[a]):k(\ell^s P)]}{[k(E[a\ell^s]):k(P)]}
            = \frac{\#H \cdot \#\left(\im \rho_{E,a}\cap \BB{a}\right)}
            {\#\left(\im \rho_{E,a\ell^s}\cap \BB{a\ell^s}\right)}\\
            & \geq \frac{\#H \cdot \#\left(\im \rho_{E,a}\cap \BB{a}\right)}
            {\#(H\cap \BB{a\ell^s})\cdot\#\left(\im \rho_{E, a} \cap \BB{a}\right)}
            = \frac{\#H}{\#(H\cap \BB{a\ell^s})} = \frac{\#H}{\#(H\cap \BB{\ell^s})}.
        \end{align*}

        {Since {$\ell\not\in S_E$}, we may use Proposition~\ref{prop:surjectiveSL} to conclude that $H$ must contain $\SL_2(\Z/\ell^s\Z).$} Therefore we have set inclusions
        \begin{equation}\label{eq:EqualityOfQuotients}
            \SL_2(\Z/\ell^s\Z)/\left(\SL_2(\Z/\ell^s\Z)\cap \BB{\ell^s}\right)\hookrightarrow
            H/\left(H\cap\BB{\ell^s}\right) \hookrightarrow 
            \Aut(E[\ell^s])/\BB{\ell^s}.
        \end{equation}

        Since the sets on the right and the left of \eqref{eq:EqualityOfQuotients} have the same cardinality, all inclusions in \eqref{eq:EqualityOfQuotients} must be bijections.  {Hence, the upper and lower bounds obtained above agree, and, in particular, $[k(P) : k(\ell^sP)] =  \#\{ Q\in E: \ell^sQ=\ell^sP,\; Q \text{ order } a\ell^s\}$ as desired.}
    \end{proof}

     \subsection{Maps between $\{X_1(n)\}$ where $n$ has specified support}\label{subsec:FixedSupport}
        \begin{prop} \label{prop:FixedSupportAndLevel}
            Let $E$ be a non-CM elliptic curve over {a number field $k$}, let $S$ be a finite set of primes, and let $\mm_S := \prod_{\ell\in S}\ell$. Let $M = M_E(S)$ be a positive integer with $\Supp(M) \subset S$ such that
            \[
            \im \rho_{E, \mm_S^{\infty}} = \pi^{-1}(\im \rho_{E, M})
            \]
            and let $a$ and $b$ be positive integers with $ \gcd(ab, M) | a$ and $\Supp(ab) \subset S$.
           {Let $x\in X_1(ab)$ be a closed point with $j(x) = j(E)$ and let $f$ denote the natural map $X_1(ab) \to X_1(a).$ Then
            \[
                \deg(x) = \deg(f)\deg(f(x)).    
            \]}
        \end{prop}
        \begin{proof}
            Let $M' := \lcm(a, M)$ and let $n = ab$.  By definition, $\im \rho_{E, n}$ is the mod $n$ reduction of $\im \rho_{E, \mm_S^{\infty}}$ and $\im \rho_{E, a}$ is the mod $a$ reduction of $\im \rho_{E, M'}$.  Since $\im \rho_{E, \mm_S^{\infty}} = \pi^{-1}(\im \rho_{E, M})$, this implies that
            \[
                \im \rho_{E, \mm_S^{\infty}} = \pi^{-1}(\im \rho_{E, M'})
                \quad \textup{ and that }\quad
                \im \rho_{E, n} = \pi^{-1}(\im \rho_{E, a}),
            \]
            where by abuse of notation, we use $\pi$ to denote both natural projections. In other words, the mod $n$ Galois representation is as large as possible given the mod $a$ Galois representation.  Hence, for any $P\in E$ of order $n$, the extension $[k(P):k(bP)]$ is as large as possible, i.e., $[k(P):k(bP)] = \# \{Q\in E : bQ = bP, Q \textup{ order }n\}$.  In particular this applies to a point $P\in E$ such that $x=[(E, P)]\in X_1(n).$ Therefore, Lemma~\ref{lem:LargeGaloisImageDegreeOfMap} completes the proof.
        \end{proof}

 	\subsection{Proof of Theorem~\ref{thm:DegreeOfMapFixedjInvariant}}
	\label{subsec:ProofOfFixedjInvariant}
       {Let $x\in X_1(n)$ be a closed point with $j(x) = j(E)$ and write $n = n_0n_1$ where $\Supp(n_0) \subset S$ and $\Supp(n_1)$ is disjoint from $S$.  Note that $\gcd(n, M)|n_0$.  We factor the map $f$ as
        \[
            X_1(n) \stackrel{f_1}{\to} X_1(n_0) \stackrel{f_2}{\to} X_1(\gcd(n,M)).     
        \]
        By inductively applying Proposition~\ref{prop:EliminatingSurjectivePrimes} to powers of primes $\ell\notin S_E$, we see that $\deg(x) = \deg(f_1)\deg(f_1(x))$.  Then we apply Proposition~\ref{prop:FixedSupportAndLevel} with $a = \gcd(n, M)$, $b = n_0/\gcd(n, M)$ to show that
        \[
            \deg(f_1(x)) = \deg(f_2)\deg(f_2(f_1(x))). 
        \]}

\section{Proof of Theorem~\ref{thm:Main}}\label{sec:Uniform}

    In this section we prove Theorem~\ref{thm:Main}. 
    For a fixed number field $k$, Conjecture~\ref{conj:Uniformity} implies that there is a finite set of primes $S = S(k)$ such that for all {non-CM} elliptic curves $E/k$, $S \supset S_{E/k}$ (see~\eqref{eq:DefnOfSE}).  Furthermore, Conjecture~\ref{conj:StrongUniformity} implies that $S(k)$ can be taken to depend only on $[k:\Q]$. Thus, to deduce Theorem~\ref{thm:Main} from Corollary~\ref{cor:SporadicPtsFixedjInvariant}, it suffices to show that for any positive integer $d$ and any finite set of primes $S$, there is an integer $M = M_d(S)$ such that for all number fields $k$ of degree $d$ and all {non-CM} elliptic curves $E/k$, we have
    \[
        \im \rho_{E, \mm_S^{\infty}} = \pi^{-1}(\im \rho_{E, M}).
    \]
    Hence Proposition~\ref{prop:UniformLevelFiniteSetPrimes} completes the proof of Theorem~\ref{thm:Main}.
    
    \begin{prop}\label{prop:UniformLevelFiniteSetPrimes}
		Let $d$ be a positive integer, $S$ a finite set of primes, and $\calE$ a set of non-CM elliptic curves over number fields of degree at most $d$.
		\begin{enumerate}
            \item\label{item1}\label{part:TheoreticalBound} There exists a  positive integer $M$ with $\Supp(M) \subset S$ such that for all $E/k \in \calE$ 
            \[
				\im \rho_{E, \frakm_S^{\infty}} = \pi^{-1}(\im \rho_{E, M}).
            \]
			\item\label{item2}\label{part:ExplicitBoundByInduction} {Let $M_d(S,\calE)$ be the smallest such $M$ as in~\eqref{item1} and for all }$\ell\in S$, define
			\[
                {\tau  = \tau_{S,\calE, \ell} := 
                \max_{E/k \in \calE} 
                \left( v_{\ell}\left(\#\im\rho_{E, \frakm_{\Snoell}}\right)\right) \leq 
                v_{\ell}(\#\GL_2(\Z/\frakm_{\Snoell}\Z)).}
			\]
			Then $v_{\ell}(M_d(S,\calE)) \leq \max(v_{\ell}(M_d(\{\ell\}, \calE)), v_{\ell}(2\ell)) + \tau$.
		\end{enumerate}
	\end{prop}
	\begin{remark}
		In the proof of Proposition~\ref{prop:UniformLevelFiniteSetPrimes}\eqref{item2}, if $\im\rho_{E,\frakm_{\Snoell}}$ is a Sylow $\ell$-subgroup of $\GL_2(\Z/\frakm_{\Snoell}\Z)$, then a chief series (a maximal normal series) of {$\im\rho_{E,\frakm_{\Snoell}}$} does have length $\tau$. So the bound in \eqref{item2} is sharp if the group structure of $\im\rho_{E,\ell^{\infty}}$ allows. However, given set values for $d, S,$ and $\calE$, information about the group structure of possible Galois representations (rather than just bounds on the cardinality) could give sharper bounds.
    \end{remark}
    \begin{remark}
        {A weaker version of Proposition~\ref{prop:UniformLevelFiniteSetPrimes} follows from~\cite[Proof of Lemma 8]{Jones09v2}.  Indeed, Jones's proof goes through over a number field and for \emph{any} finite set of primes $S$ (rather than only $S = \{2,3,5\}\cup\{p: \im \rho_{E,p} \neq \GL_2(\Z/p\Z)\}\cup \Supp(\Delta_E)$, {which is the case under consideration in~\cite[Lemma 8]{Jones09v2}}) and shows that 
        \[
            v_{\ell}(M_d(S,\calE)) \leq \max(v_{\ell}(M_d(\{\ell\}, \calE)), v_{\ell}(2\ell)) + v_{\ell}(\#\GL_2(\Z/\frakm_{\Snoell}\Z)).
        \]
        The proof here and the one in~\cite{Jones09v2} roughly follow the same structure; however, by isolating the purely group-theoretic components (e.g., Proposition~\ref{prop:levelforcomp}), we are able to obtain a sharper bound in~\eqref{part:ExplicitBoundByInduction}.}
    \end{remark}
    \begin{proof}
        When $\#S = 1$, part~\eqref{part:TheoreticalBound} follows from Theorem~\ref{thm:UniformIndex} and Proposition~\ref{prop:MaximalKernelGivesLevel} and part~\eqref{part:ExplicitBoundByInduction} is immediate.

        {We prove part~\eqref{part:TheoreticalBound} when $\#S$ is arbitrary} by induction using Proposition~\ref{prop:levelforcomp}.  Let $S = \{\ell_1, \dots, \ell_q\}$, let $M_i =  M_d(\{{\ell_i}\}, \calE)$, let $s_i = \max(v_{\ell_i}(M_i), v_{\ell_i}(2\ell_i))$, and let $N_i = \prod_{j\neq i}\ell_j^{s_j}$.  It suffices to show that for all $1\leq i\leq q$ there exists a $t_i \geq s_i$ such that for all $E/k\in \calE$
        \[
            \im \rho_{E, N_i\cdot \ell_i^{\infty}} = \pi^{-1}(\im \rho_{E, N_i\ell_i^{t_i}});  
        \]
        then Proposition~\ref{prop:levelforcomp} implies that we may take $M = \prod_{i}\ell_i^{t_i}.$

        Fix {$i\in \{1,\dots, q\}$}.  For any $E/k\in\calE$ and any $s\geq s_i$, define
        \[
            K_{E,s}^i :=  \ker(\im \rho_{E, N_i\cdot \ell_i^s} \to \im \rho_{E, N_i}), \quad \textup{and} \quad
            L_{E,s}^i :=  \ker(\im \rho_{E, N_i\cdot \ell_i^s} \to \im \rho_{E, \ell_i^s}).
        \]
        By definition, $K^i_{E, s'}$ maps surjectively onto $K^i_{E,s}$ for any $s'\geq s$, so $K^i_{E,s}$ is the mod $N_i\ell_i^s$ reduction of $K_E^i := \ker(\im \rho_{E, N_i\cdot {\ell_i^{\infty}}} \to \im \rho_{E, N_i})$.  Let us now consider $L^i_{E,s}$.  Since $\ell_i\nmid N_i$, $L^i_{E,s}$ can be viewed as a subgroup of $\im\rho_{E, N_i}$ and {we have} $L^i_{E,s'}\subset L^i_{E,s}$ for all $s'\geq s$. Let $r\geq s_i$ be an integer such that $L^i_{E,r} = L^i_{E, r+1}$.
        Then we have the following diagram
        \begin{equation}\label{eq:diag2}
        \xymatrix{
            \im \rho_{E, \ell_i^{r+1}}/K^i_{E,r+1} \ar@{->>}[r] \ar[d]^{\isom} &  \im \rho_{E, \ell_i^{r}}/K^i_{E,r} \ar[d]^{\isom} \\
            \im \rho_{E, N_i} / L^i_{E,r+1} \ar@{=}[r]
             & \im \rho_{E, N_i} / L^i_{E,r}}
        \end{equation}
        where the vertical isomorphisms are given by Goursat's Lemma (Lemma~\ref{lem:Goursat})\footnote{By tracing through the isomorphism given by Goursat's lemma, one can prove that this diagram is commutative.  We do not do so here, since the claims that follow can also be deduced from cardinality arguments.}. Since $r\geq s_i$,  $\im \rho_{E, \ell_i^{r+1}}$ is the full preimage of $\im\rho_{E, \ell_i^r}$ under the natural reduction map. So \eqref{eq:diag2} implies that $K^i_{E,r+1}$ is the full preimage of $K^i_{E, r}$ under the natural reduction map. Then by Proposition~\ref{prop:MaximalKernelGivesLevel}, $K^i_E$ is the full preimage of $K^i_{E, r}$ under the map $\GL_2(\Z_\ell) \to \GL_2(\Z/\ell^r\Z)$ and therefore $\im \rho_{E, N_i\cdot\ell_i^\infty} = \pi^{-1} (\im\rho_{E, N_i\ell_i^r})$.   Hence we may take {$t_{E,i}$} to be the minimal {$r\geq s_i$} such that $L^i_{E,r} = L^i_{E, r+1}$. Since $L^i_{E,s}$ is a subgroup of $\im\rho_{E, N_i} \subset \GL_2(\Z/N_i\Z)$, {$t_{E,i}$} may be bounded independent of $E/k$, depending only on $N_i$. This completes the proof of~\eqref{part:TheoreticalBound}.
            
        It remains to prove~\eqref{part:ExplicitBoundByInduction}.  Let $s\geq s_i$ and consider the following diagram, where again the vertical isomorphisms follow from Goursat's Lemma.
        \begin{equation}\label{eq:diag3}
            \xymatrix{
                \im \rho_{E, \ell_i^{s}}/K^i_{E,s} \ar@{->>}[r] \ar[d]^{\isom} &  \im \rho_{E, \ell^{s_i}}/K^i_{E, s_i} \ar[d]^{\isom} \\
                \im \rho_{E, N_i} / L^i_{E,s} \ar@{->>}[r] & \im \rho_{E, N_i} / L^i_{E,s_i}
            }
        \end{equation}
        The kernel of the top horizontal map is an $\ell_i$-primary subgroup, so the index of $L^i_{E,s}$ in $L^i_{E,s_i}$ is a power of $\ell_i$. Thus, the maximal chain of proper containments $L^i_{E, s_i}\supsetneq L^i_{E, s_i+1}\supsetneq \dots\supsetneq L^i_{E, t_i}$ is bounded by $v_{\ell_i}({\#}\im \rho_{E, N_i}) = v_{\ell_i}({\#}\im \rho_{E, \frakm_{{S-\{\ell_i\}}}})$, 
        which yields~\eqref{part:ExplicitBoundByInduction}.
    \end{proof}

\section{Lifting sporadic points}\label{sec:LiftingSporadic}
    {In this section we study when a sporadic point on $X_1(n)$ lifts to a sporadic point on a modular curve of higher level.}  
    We give a numerical criterion that is sufficient for lifting sporadic points (see Lemma~\ref{lem:Lifting}), and use this to prove that 
    there exist sporadic points such that \emph{every} lift is sporadic. The examples we have identified correspond to {CM elliptic curves}.

    \begin{theorem} \label{thm:CM}
        Let $E$ be an elliptic curve with CM by an order in an imaginary quadratic field $K$. Then for all sufficiently large primes $\ell$ which split in $K$, there exists a sporadic point $x=[(E,P)] \in X_1(\ell)$ with only sporadic lifts. Specifically, for any positive integer $d$ and any point $y \in X_1(d\ell)$ with $\pi(y) = x$, the point $y$ is sporadic, where $\pi$ denotes the natural map $X_1(d\ell) \rightarrow X_1(\ell)$. 
    \end{theorem}

    The key to the proof of Theorem \ref{thm:CM} is producing a sporadic point of sufficiently low degree so we may apply the following lemma. It is a consequence of Abramovich's lower bound on gonality in \cite{Abramovich-Gonality} and the result of Frey \cite{Frey-InfinitelyManyDegreed} which states that a curve $C_{/K}$ has infinitely many points of degree at most $d$ only if $\gon_K(C) \leq 2d$.

    \begin{lemma}\label{lem:Lifting}
        {Suppose there is a point $x \in X_1(N)$ with 
        \[
        \deg(x) < \frac{7}{1600}[\PSL_2(\Z) : \Gamma _1(N)].
        \]
        {Then $x$ is sporadic and} for any positive integer $d$ and any point $y \in X_1(dN)$ with $\pi(y) = x$, the point $y$ is sporadic, where $\pi$ denotes the natural map $X_1(dN) \rightarrow X_1(N)$.}
    \end{lemma}

    \begin{proof}
        {We claim that the assumption on the degree of $x$ implies that $\deg(y) < \frac{7}{1600}[\PSL_2(\Z) : \Gamma _1(dN)]$.  Then~\cite[Thm. 0.1]{Abramovich-Gonality}, shows that 
        \[
            \deg(y) < \frac{1}{2}\gon_{\mathbb{Q}}(X_1(d N))\quad\textup{and}\quad
            \deg(x) < \frac{1}{2}\gon_{\mathbb{Q}}(X_1(N)).
        \]
        Thus $x$ and $y$ are sporadic by \cite[Prop. 2]{Frey-InfinitelyManyDegreed}.}
        
        Now we prove the claim.  Let $x \in X_1(N)$ be a $ \deg(x) \leq \frac{7}{1600}[\PSL_2(\Z) : \Gamma _1(N)]$. In particular, this implies $N>2$. Thus for any point $y \in X_1(dN)$ with $\pi(y)=x$ we have
        \begin{align*}
        \deg(y) &\leq \deg(x) \cdot \deg(X_1(dN) \rightarrow X_1(N))\\
        &< \dfrac{7}{1600} [\PSL_2(\Z) : \Gamma _1(N)] \cdot d^2 \prod_{p \mid d, \, p \nmid N} \left(1 - \frac{1}{p^2} \right) \, \, \, \, \, \text{(see Proposition \ref{prop:Degree})}\\
        &= \dfrac{7}{1600}\cdot \frac{1}{2}(dN) \prod_{p \mid dN} \left(1 + \frac{1}{p} \right) \varphi(dN) \\
        &=\dfrac{7}{1600} [\PSL_2(\Z) : \Gamma _1(dN)]. 
                                \end{align*}
\end{proof}

\begin{proof}[Proof of Theorem~\ref{thm:CM}]
Let $E$ be an elliptic curve with CM by an order $\mathcal{O}$ in an imaginary quadratic field $K$. Then $L :=  K(j(E))$ is the ring class field of $\mathcal{O}$ and $[L:K] = h(\mathcal{O})$, the class number of $\mathcal{O}$. (See \cite[Thms. 7.24 and 11.1]{Cox} for details.)
Let $\ell$ be a prime that splits in $K$ 
and satisfies
\[ \ell >  \left(\dfrac{6400}{7} \cdot \dfrac{h(\mathcal{O}) }{\# \mathcal{O} ^{\times}}\right) - 1.  \]
By~\cite[Thm. 6.2]{BourdonClark}, there is a point $P \in E$ of order $\ell$ with
\[ [L(\mathfrak{h}(P)):L] = \frac{\ell - 1}{\# \mathcal{O} ^{\times}} .\] Then for $x = [(E, P)] \in X_1(\ell)$,
\begin{align*}
\deg (x) = [\Q(j(E), \mathfrak{h}(P)): \Q] &\leq [K(j(E), \mathfrak{h}(P)): \Q] = [L(\mathfrak{h}(P)):\Q]
 = \frac{\ell - 1}{\# \mathcal{O} ^{\times}}  \cdot h(\mathcal{O}) \cdot2\\
&< (\ell - 1) \cdot \dfrac{7}{6400}(\ell + 1) \cdot 2  
= \dfrac{7}{1600} [\PSL_2(\Z) : \Gamma _1(\ell )].
\end{align*}
The result now follows from Lemma~\ref{lem:Lifting}.
\end{proof}

\begin{remark}
Note that none of the known non-cuspidal non-CM sporadic points satisfy the degree condition given in Lemma~\ref{lem:Lifting}.  Thus it is an interesting open question to determine 
whether there exist non-CM sporadic points with infinitely many sporadic lifts. If no such examples exist, then by Theorem \ref{thm:Main} there would be only finitely many non-CM sporadic points corresponding to $j$-invariants of bounded degree, assuming Conjecture \ref{conj:StrongUniformity}.
\end{remark}

\section{{Isolated} points with rational $j$-invariant}
\label{sec:RationaljInvariant}
  In this section, we study non-CM {isolated} points with rational $j$-invariant. {Our main result of this section (Theorem~\ref{thm:classification}) gives a classification of the non-cuspidal non-CM  {isolated} points on $X_1(n)$ with rational $j$-invariant.  We prove that they either arise from elliptic curves whose Galois representations are very special (and may not even exist), or they can be mapped to {isolated} points on $X_1(m)$ for an explicit set of integers $m$. }
 
  Later, we focus on sporadic points with rational $j$-invariant on $X_1(n)$ for particular values of $n$.  We show that if $n$ is prime (Proposition~\ref{prop:Primes}), is a power of $2$ (Proposition~\ref{prop:2primary}), or, conditionally on Sutherland~\cite[Conj. 1.1]{Sutherland-ComputingGaloisImages} and Zywina~\cite[Conj. 1.12]{Zywina-PossibleImages}), has $\min(\Supp(n))\geq 17$ (Proposition~\ref{prop:LevelDivisibleByLargePrime}), then any non-CM, non-cuspidal sporadic point with rational $j$-invariant has $j(x) = -7\cdot 11^3$.

\subsection{Classification of non-CM isolated points with rational $j$-invariant}\label{subsec:classification}
    \begin{theorem}\label{thm:classification}
    Let $x\in X_1(n)$ be a non-CM non-cuspidal {isolated} point with $j(x)\in \Q$.  Then one of the following holds:
    \begin{enumerate}
      \item There is an elliptic curve $E/\Q$ with $j(E) = j(x)$ and {a prime $\ell\in \Supp(n)$ such that either $\ell> 17, \ell\neq 37$ and $\rho_{E, \ell}$ is not surjective or $\ell = 17$ or $37$ and $\rho_{E, \ell}$ is a subgroup of the normalizer of a non-split Cartan {subgroup}.}  \label{case:LargePrimeNonSurjective} 
      \item There is an elliptic curve $E/\Q$ with $j(E) = j(x)$ and two distinct primes $\ell_1 > \ell_2 > 3$ {in $\Supp(n)$ }such that both $\rho_{E, \ell_1}$ and $\rho_{E, \ell_2}$ are not surjective.\label{case:TwoPrimesNonSurjective}
      \item There is an elliptic curve $E/\Q$ with $j(E) = j(x)$ and a prime $2<\ell \leq 37$ {in $\Supp(n)$} such that the $\ell$-adic Galois representation of $E$ has level greater than $169$.\label{case:LargeLevel}
      \item {There is a divisor of $n$ of the form $2^a3^bp^c$ such that the image of $x$ in $X_1(2^a3^bp^c)$ is {isolated} and such that $a\leq a_{p}$, $b\leq b_{p}$, $p^c\leq 169$ for one of the following values of $p$, $a_{p}$, $b_{p}$.}
      \[
	        \begin{tabular}{c||c|c|c|c|c|c|c}
                $p$ & 1& 5 & 7 & 11 & 13 & 17 & 37 \\
                \hline
                $a_{p}$ & 9 & 14 & 14 &  13 &  14 &  15 &  13  \\
                \hline
                $b_{p}$ & 5 & 6 & 7 & 6 &  7 &  5 &  8
	        \end{tabular}
      \]\label{case:SporadicLevels}
    \end{enumerate}
  \end{theorem}
{\begin{remark}
  This theorem also holds for $x$ a $\Pisolated$, $\AVisolated$, or sporadic point, respectively, at the expense of~\eqref{case:SporadicLevels} giving the statement that the image of $x$ is $\Pisolated$, $\AVisolated$, or sporadic, respectively.
\end{remark}}
\begin{remark}\label{rmk:UnlikelyGaloisRep}
  Each of cases $\eqref{case:LargePrimeNonSurjective}$, $\eqref{case:TwoPrimesNonSurjective}$, and $\eqref{case:LargeLevel}$ should be rare situations, if they occur at all. Indeed, the question of whether elliptic curves as in~\eqref{case:LargePrimeNonSurjective} exist is related to a question originally raised by Serre in 1972, and their non-existence has since been conjectured by Sutherland~\cite[Conj. 1.1]{Sutherland-ComputingGaloisImages} and Zywina~\cite[Conj. 1.12]{Zywina-PossibleImages}.  
  
Assuming $\eqref{case:LargePrimeNonSurjective}$ does not hold, elliptic curves as in~\eqref{case:TwoPrimesNonSurjective} correspond to points on finitely many modular curves of genus greater than $2$, so there are at worst finitely many $j$-invariants in this case \cite[Tables 6--14, Theorem 16A]{DanielsGonzalezJimenez}.   Additionally, there are no elliptic curves in the LMFDB database~\cite{LMFDB} as in~\eqref{case:TwoPrimesNonSurjective}, so in particular, any elliptic curve as in~\eqref{case:TwoPrimesNonSurjective} must have conductor larger than $400,000.$  (The Galois representation computations in LMFDB were carried out using the algorithm from~\cite{Sutherland-ComputingGaloisImages}.)  {If we do not assume~\eqref{case:LargePrimeNonSurjective} does not hold, then we must consider the case where $\ell_1>37$.  In this case the elliptic curves of interest no longer correspond to points on finitely many modular curves, but nevertheless, Lemos has shown that such elliptic curves do not exist, assuming that $\im \rho_{E,\ell_2}$ is contained in the normalizer of a split Cartan subgroup or in a Borel subgroup~\cites{Lemos-Borel, Lemos-CartanNS}.}
  
  Sutherland and Zywina's classification of modular curves of prime-power level with infinitely many points~\cite{SZ-primepower} shows that there are only finitely many rational $j$-invariants corresponding to elliptic curves as in~\eqref{case:LargeLevel}, and suggests that in fact they do not exist.  Table~\ref{table:SZsummary} gives, for each prime $\ell$, the maximal prime-power level for which there exists a modular curve of that level with infinitely many rational points.
  \begin{table}[h]
  \begin{tabular}{c||c|c|c|c|c|c|c}
    $\ell$ & 3 & 5 & 7 & 11 & 13 & 17 & 37 \\
    \hline
    max level & 27 & 25 & 7 &  11 &  13 &  1 &  1  \\
  \end{tabular}\caption{Maximal prime-power level for which there exists a modular curve with infinitely many rational points}\label{table:SZsummary}
\end{table}
Therefore, for $3\leq \ell\leq 37$, respectively, there are already only finitely many $j$-invariants of elliptic curves with an $\ell$-adic Galois representation of level at least $81$, $125$, $49$, $121$, $169$, $17$, or $37$.  Since such $j$-invariants are already rare, it seems reasonable to expect any such correspond to elliptic curves of $\ell$-adic level \emph{exactly}  $81$, $125$, $49$, $121$, $169$, $17$ and $37$, respectively.

This has been (conditionally) verified by Drew Sutherland in the cases $\ell = 17$ and $\ell = 37$.  For these primes, there are conjecturally only $4$ $j$-invariants corresponding to elliptic curves with non-surjective $\ell$-adic Galois representation: $-17\cdot 373^3/2^{17}, -17^2 \cdot 101^3/2, -7 \cdot 11^3,$ and $-7 \cdot 137^3 \cdot 2083^3$~\cite[Conj. 1.12]{Zywina-PossibleImages}.  For each of these $j$-invariants, Sutherland computed that the $\ell$-adic Galois representation is the full preimage of the mod $\ell$ representation, so the representations are indeed of level $\ell$ and not level $\ell^2$ \cite{Sutherland-PersonalCommunication}.\footnote{Sutherland used a generalization of the algorithm in~\cite{Sutherland-ComputingGaloisImages} to prove in each case the index of the mod-$\ell^2$ image is no smaller than that of the mod-$\ell$ image.  It then follows from~\cite[Lemma 3.7]{SZ-primepower} that  the $\ell$-adic image is the full preimage of the mod-$\ell$ image.}

\end{remark}

\begin{proof}
    Let $x\in X_1(n)$ be a non-cuspidal non-CM {isolated} point with $j(x)\in \Q$. Let $E$ be an elliptic curve over {$\Q$} with $j(E)=j(x)$. Assume that~\eqref{case:LargePrimeNonSurjective} does not hold, so in particular $E$ has surjective mod $\ell$ representation for every $\ell > 17$ and $\ell \neq 37$.  Thus Proposition~\ref{prop:EliminatingSurjectivePrimes} {and Theorem~\ref{thm:PushingForwardSporadicIsolated}} together imply that $x$ maps to an {isolated} point on $X_1(n')$ where $n'$ is the largest divisor of $n$ that is not divisible by any primes greater than $17$ except possibly $37$.

    Now assume further that~\eqref{case:TwoPrimesNonSurjective} does not hold, so there is at most one prime $p>3$ for which the $p$-adic Galois representation is not surjective.  If the $p$-adic Galois representation of $E$ is surjective for all primes larger than $3$, then we will abuse notation and set $p=1$.  Under these assumptions, additional applications of Proposition~\ref{prop:EliminatingSurjectivePrimes} {and Theorem~\ref{thm:PushingForwardSporadicIsolated}} show that $x$ maps to an {isolated} point on $X_1(n'')$ where $n''$ is a divisor of $n'$ with $\Supp(n'') \subset S := \{2,3,p\}$\footnote{When $p = 1$, we conflate the set $\{2,3,p\}$ with the set $\{2,3\}$.} and $p \in \{1,5,7,11,13,17,37\}$.  
    Furthermore, Corollary~\ref{cor:SporadicPtsFixedjInvariant} shows that $x$ maps to an {isolated} point on $X_1(\gcd(n'', M))$, where $M$ is the level of the {$\mm_S^{\infty}$} Galois representation of $E$.

    Now we will further assume that~\eqref{case:LargeLevel} does not hold. Let $\calE$ denote the set of all {non-CM} elliptic curves over $\Q$.   Proposition~\ref{prop:UniformLevelFiniteSetPrimes} states that there is an integer {$M_1(S, \calE)$} such that the level of the  {$\mm_S^{\infty}$} Galois representation of $E$ divides {$M_1(S, \calE)$ for all $E\in \calE$}.  We will show that {$M_1(\{2,3, p\}, \calE)$} divides $2^{a_{p}}3^{b_{p}}p^c$ for $p, a_{p}, b_{p}, c$ as in~\eqref{case:SporadicLevels}.

    By the assumption that~\eqref{case:LargeLevel} does not hold and~\cite[Corollary~1.3]{RouseZureickBrown}, we have the following values for the constant $M_1(\{\ell\}, \calE)$ from Proposition~\ref{prop:UniformLevelFiniteSetPrimes}.
   \begin{center}
        \begin{tabular}{c||c|c|c|c|c|c|c|c}
                 $\ell$ & 2 & 3 & 5 & 7 & 11 & 13 & 17 & 37 \\
            \hline
            $M_1(\{\ell\}, \calE)$ & $2^5$ & $3^4$ & $5^3$ & $7^2$ &  $11^2$ &  $13^2$ &  17 &  37  
        \end{tabular}
    \end{center}

    By Proposition~\ref{prop:UniformLevelFiniteSetPrimes}\eqref{part:ExplicitBoundByInduction},
    \[
        v_{\ell}(M_1(S, \calE)) \leq  \max(v_{\ell}(M_1(\{\ell\}, \calE)), v_{\ell}(2\ell)) + \sum_{\ell'\in \Snoell} v_{\ell}(\#\GL_2(\Z/\ell'\Z)).
    \]
    This upper bound combined with Table~\ref{table:SizeGL2} yields the desired divisibility except for the case where $p = 17$ or $p=37$.  

    \begin{table}[h]
      \begin{tabular}{c||c|c|c|c|c|c|c|c}
               $\ell$ & 2 & 3 & 5 & 7 & 11 & 13 & 17 & 37 \\
          \hline\\[-1em]
          $\#\GL_2(\Z/\ell\Z)$ & $2\cdot 3$ & $2^43$ & $2^53^1 5$ & $2^53^27$ &  $2^43^15^211$ &  $2^53^27^113$ &  $2^93^217$ &  $2^53^419^137$ 
      \end{tabular}
      \vspace{1em}\caption{Cardinality of $\GL_2(\Z/\ell\Z)$}\label{table:SizeGL2}
  \end{table}

    Let us consider the case that $p = 17$, {so $\rho_{E, 17}$ is not surjective}.   
    Since we are not in case~\eqref{case:LargePrimeNonSurjective}, we know $\im\rho_{E,17}$ is not contained in the normalizer of the non-split Cartan.  Thus~\cite[Thms. 1.10 and 1.11]{Zywina-PossibleImages} show that $\#\im\rho_{E,17} = 2^617$, so Proposition~\ref{prop:UniformLevelFiniteSetPrimes}\eqref{part:ExplicitBoundByInduction} implies that the level of the $\mm_S^{\infty}$ Galois representation divides $2^{15} 3^5 17$.

    The case when $p=37$ proceeds similarly.  In this case~\cite[Thms. 1.10 and 1.11]{Zywina-PossibleImages} show that $\#\im\rho_{E, 37} = 2^4 3^3 37$ and so Proposition~\ref{prop:UniformLevelFiniteSetPrimes}\eqref{part:ExplicitBoundByInduction} implies that the level of the $\mm_S^{\infty}$ Galois representation divides $2^{13} 3^8 37$.
  \end{proof}

\subsection{Rational $j$-invariants of non-CM non-cuspidal sporadic points on $X_1(n)$ for particular values of $n$}\label{subsec:SporadicClassification}
  
  \begin{prop}\label{prop:Primes}
    Fix a prime $\ell$.  If $x\in X_1(\ell)$ is a non-CM non-cuspidal sporadic point with $j(x)\in \Q$ then $\ell = 37$ and $j(x) = -7\cdot11^3$. 
  \end{prop}
  \begin{proof}
    Let $x=[(E,P)]$ be a non-CM sporadic point on $X_1(\ell)$ with $j(E) \in \Q$.  We may assume $E$ is defined over $\Q$.   Note that $X_1(\ell)$ has infinitely many rational points for $\ell \leq 10$.  Further, $X_1(\ell)$ has gonality $2$ for $\ell=11, 13$, and no non-cuspidal rational points~\cite{Mazur}.  Hence if $x\in X_1(\ell)$ is a non-cuspidal {non-CM} sporadic point, $\ell > 13$.

    If the mod $\ell$ Galois representation of $E$ is surjective, then $x$ cannot be a sporadic point on $X_1(\ell)$ by Corollary \ref{cor:SurjectivePrimes}, so assume that $\rho_{E,\ell}$ is not surjective.  Then the {im $\rho_{E,\ell}$} is contained in a maximal subgroup, which can be an exceptional subgroup, a Borel subgroup or the normalizer of a (split or non-split) Cartan subgroup of $\GL_2(\F_{\ell})$ \cite[Section 2]{Serre-OpenImage}. We will analyze each case separately.
  
    In the case {where im $\rho_{E,\ell}$} is contained in the normalizer of the non-split Cartan subgroup, Lozano-Robledo \cite[Theorem 7.3]{LR-TorsionFieldOfDefn} shows that the degree of a field of definition of a point of order $\ell$ is greater than or equal to $(\ell^2-1)/6$. Since {$\ell>13$ we have
        \[
      \text{gon}_{\Q}(X_1(\ell)) \leq \text{genus}(X_1(\ell))  \leq \frac{1}{24}(\ell^2-1). 
    \] 
    Therefore $x$ cannot be sporadic in this case.}

    If {im $\rho_{E,\ell}$} is contained in the normalizer of the split Cartan subgroup, then by \cite{Split Cartan}, $\ell$ has to be less than or equal to $13$.  Similarly, if {im $\rho_{E,\ell}$} is one of the exceptional subgroups, then by \cite[Theorem 8.1]{LR-TorsionFieldOfDefn}, $\ell \leq 13$.

    If {im $\rho_{E,\ell}$} is contained in a Borel subgroup, then $E$ has a rational isogeny of degree $\ell$. By \cite{Mazur-Isogenies}, $\ell$ is one of the following primes: $2,3,5,7,11,13,17,37$. {Thus we need only consider $\ell = 17$ and $37$.} For $\ell=17$, \cite[Table 5]{LR-TorsionFieldOfDefn} shows that {$\deg(x)\geq 4$}. Since the gonality of $X_1(17)$ is also $4$, {$x$ cannot} be sporadic. 
  
    Finally when $\ell=37$, there are exactly two non-cuspidal points in $X_0(37)(\Q)$ \cite[Table 5]{LR-TorsionFieldOfDefn}. The one corresponding to an elliptic curve with $j$-invariant $-7\cdot 11^3$ gives a degree $6$ point on $X_1(37)$, {which is sporadic since $\gon_{\Q}X_1(37)=18$.} The other gives a point on $X_1(37)$ of degree 18, which is not sporadic.
  \end{proof} 

 \begin{prop}\label{prop:2primary}
   Let $s\geq 1$.  If $x\in X_1(2^s)$ is a non-cuspidal non-CM sporadic point, then $j(x) \notin \Q$.
  \end{prop}
  \begin{proof}
    By~\cite[Cor. 1.3]{RouseZureickBrown}, the $2$-adic Galois representation of any non-CM elliptic curve over $\Q$ has level at most $32$.  Thus, by Proposition~\ref{prop:FixedSupportAndLevel} it suffices to show that $X_1(2^s)$ has no non-cuspidal non-CM sporadic points with rational $j$-invariant for $s \leq 5$.

    If $s=1,2$ or $3$, then modular curve $X_1(2^a)$ is isomorphic to $\PP^1_{\Q}$ and so has no sporadic points. When $s = 4$, the modular curve $X_1(16)$ has genus $2$ and hence gonality $2$ which implies that it has infinitely many points of degree $2$. Additionally, as first established by Levi \cite{Levi}, $X_1(16)$ has no {non-cuspidal} points over $\Q$ and so has no {non-cuspidal} sporadic points.

    Now we consider $X_1(32)$, which has gonality $8$ (see \cite[Table 1]{DerickxVanHoeij}). Let $x=[(E,P)]$ be a non-CM sporadic point on $X_1(32)$ with $j=j(E)\in \Q$. We may assume that $E$ is defined over $\Q$. 
    Since $x$ is a sporadic point, there are only finitely many points $y\in X_1(32)$ with $\deg(y) \leq \deg(x)$. 
    Since the degree of a point $y\in X_1(32)$ {can be calculated from} the mod $32$ Galois representation of an elliptic curve with $j$-invariant $j(y)$, 
    this implies that there are only finitely many $j$-invariants whose mod $32$ Galois representation is contained in a conjugate of {im $\rho_{E, 32}$}.
    By~\cite[Table~1]{RouseZureickBrown}, there are only eight non-CM $j$-invariants with this property: 
    \[
      2^{11}, \; 2^417^3, \; \frac{4097^3}{2^4},\;
\frac{257^3}{2^8}, \;
-\frac{857985^3}{62^8},\;
\frac{919425^3}{496^4},\;
-\frac{3\cdot18249920^3}{171^6},\;\textup{and}\;
-\frac{7\cdot1723187806080^3}{79^{16}}.
    \] 
    Using \texttt{Magma}, we compute the degree of each irreducible factor of $32^{\textup{nd}}$ division polynomial for each of these $j$-invariants and we find that the least degree of a field where a point of order $32$ is defined is $32$, hence there are no non-CM sporadic points on $X_1(32)$ with a rational $j$-invariant.
  \end{proof}
 
\begin{prop}\label{prop:LevelDivisibleByLargePrime}
    Let $n$ be a positive integer with $\min(\Supp(n)) \geq 17$.  Assume~\cite[Conj. 1.1]{Sutherland-ComputingGaloisImages} or~\cite[Conj. 1.12]{Zywina-PossibleImages}. If $x\in X_1(n)$ is a non-cuspidal non-CM sporadic point with $j(x) \in \Q$, then $37|n$ and $j(x) = -7\cdot 11^3$.
\end{prop}
\begin{proof}
    Let $E$ be an elliptic curve over $\Q$ with $j(E) = j(x)$.  We apply Theorem~\ref{thm:classification}.  By assumption and Remark~\ref{rmk:UnlikelyGaloisRep}, cases~\eqref{case:LargePrimeNonSurjective} and~\eqref{case:LargeLevel} of Theorem~\ref{thm:classification} do not occur.  Further, case~\eqref{case:TwoPrimesNonSurjective} only occurs if $17\cdot 37| n$ and $\im \rho_{E, 17}$ and $\im \rho_{E, 37}$ are both contained in Borel subgroups {(see proof of Proposition \ref{prop:Primes})}, which is impossible (see, e.g., {\cite[Table~4]{LR-TorsionFieldOfDefn}}).  

    Hence, we must be in case~\eqref{case:SporadicLevels} of Theorem~\ref{thm:classification}.  Since $\min(\Supp(n))\geq 17$, the only possible divisors of $n$ of the form $2^a3^bp^c$ (with $a, b, c, p$ as in Theorem~\ref{thm:classification}\eqref{case:SporadicLevels}) are $17$ or $37$.  Thus, for one of $\ell = 17$ or $37$ we must have $\ell | n$ and $x$ maps to a sporadic point on $X_1(\ell)$.
    Proposition~\ref{prop:Primes} then completes the proof.
\end{proof}

\end{document}